\documentclass[journal,twoside,web]{ieeecolor}
\usepackage{generic}
\usepackage{amsmath,amssymb,amsfonts}
\usepackage{algorithmic}
\usepackage[hidelinks]{hyperref}
\usepackage{subcaption}
\usepackage{graphicx}
\usepackage{textcomp}
\usepackage[ruled]{algorithm2e}

\usepackage[style=ieee,backend=biber,url=false,doi=false,doi=false,isbn=false,date=year,citestyle=numeric-comp,maxbibnames=3,minbibnames=3,maxcitenames=3,mincitenames=3]{biblatex}

\def\BibTeX{{\rm B\kern-.05em{\sc i\kern-.025em b}\kern-.08em
		T\kern-.1667em\lower.7ex\hbox{E}\kern-.125emX}}
\newtheorem{theorem}{Theorem}
\newtheorem{corollary}{Corollary}
\newtheorem{assumption}{Assumption}

\newtheorem{remark}{Remark}
\newtheorem{lemma}{Lemma}

\newtheorem{example}{Example}

\newcommand{\Ae}[1]{\color{black} #1 \color{black}}
\newcommand{\Aee}[1]{\color{black} #1 \color{black}}

\newcommand{\TF}[1]{\color{black}  #1 \color{black}}

\bibliography{alexLit,Aggregation}

\begin{document}
	\title{ Approximate  Dynamic Programming  with Feasibility Guarantees  }
	\author{Alexander Engelmann, Maísa Beraldo Bandeira, Timm Faulwasser
		\thanks{AE and MBB acknowledge financial support by the German Federal Ministry for Economic Affairs and Climate Action (BMWK) under agreement no. 03EI4043A~(Redispatch 3.0).}
	\thanks{All authors are with the Institute for Energy Systems, Energy Efficiency and Energy Economics, TU Dortmund University, Dortmund, Germany. (e-mail: \{alexander.engelmann, timm.faulwasser\}@ieee.org, maisa.bandeira@tu-dortmund.de).}}
	\maketitle
	
	\begin{abstract}
		Safe and economic operation of networked systems is often challenging. Optimization-based  schemes are frequently considered, since they achieve near-optimality while ensuring safety via the explicit consideration of constraints. In applications, these schemes, however, often require solving large-scale optimization problems. Iterative techniques from distributed optimization are frequently proposed for complexity reduction. Yet, they achieve feasibility only asymptotically, which induces a substantial computational burden. This work presents an approximate dynamic programming scheme, which is guaranteed to deliver a feasible solution in “one shot”, i.e., in one backward-forward iteration over all subproblems provided they are coupled by a tree structure. Our proposed scheme generalizes   methods from seemingly disconnected domains such as power systems and optimal control. We demonstrate its efficacy for problems with nonconvex constraints via numerical examples from both domains.
	\end{abstract}
	

\newcommand{\assumptionautorefname}{Assumption}

\section{Introduction}
\Aee{Solving large-scale optimization problems for  networked systems  is inherently complex and challenging.} 
\Ae{These problems arise} frequently in the context of critical infrastructure networks such as power systems \cite{Kim1997,Erseghe2014}, district heating and HVAC systems \cite{Ferro2022,   Camponogara2021},  or water networks \cite{Coulbeck1988}.
\Ae{At the same time, they often exhibit tree structures} induced by \Ae{the physical network design and by hierarchical operation principles.}
Stochastic optimization \cite{Shapiro2014,Hubner2020,Pacaud2022}, Model Predictive Control (MPC) \cite{Kozma2013,Shin2019a,Jiang2021a}, and
 combinations thereof \cite{Kouzoupis2018}  \TF{carry similar graph structures, whereby vertices refer to discrete time steps (which implies a path graph) or to realizations of random variables (which leads to a tree).} 
Moreover, methods with a low number of communication rounds are \Ae{preferred in order to reduce algorithm complexity and to avoid  negative effects of communication and computation delays.}

One line of research addresses the above challenges via distributed optimization, i.e.  by decomposition of the overall problem into multiple \Ae{(simpler/smaller) local instances, which} are repeatedly solved until \Ae{convergence and consensus are} achieved \cite{Kim2000,Erseghe2015, Boyd2011,Engelmann2020b, Engelmann2021a,Houska2016,Bertsekas1997a,Zavala2008a}.
\Ae{This way, complexity is reduced while feasibility holds asymptotically.}
Most of the classic distributed algorithms are guaranteed to converge for convex problems \Ae{while many of the above problems are  nonconvex.}

\Ae{For tree-structured problems,} an alternative to distributed optimization are hierarchical schemes based on Dynamic Programming~(DP)  \cite{Bellman1957,Bertele1972,Esogbue1974,Powell2011,Vandenberghe2015,Jiang2021}.
\TF{Classic DP schemes  exploit graph structures \Ae{of optimization problems} by considering subproblems at each vertex via optimal value functions. 
This substantially reduces the number of decision variables per subproblem.}
\Ae{The value functions are computed backwards in recursive fashion  from the leaves to the root.
In a forward pass, the optimal coupling variables between vertices
are computed by solving the corresponding optimization problems at each vertex. }
An advantage of \Ae{many} DP  \TF{schemes is that a solution is obtained in  \emph{one shot}, i.e., one backward-forward pass over all subproblems suffices.}
\Ae{However, due to the infamous curse of dimensionality,  computing  value functions remains difficult.}
Function approximations mitigate this difficulty leading  to Approximate Dynamic Programming~(ADP) ~\cite{Powell2011,Bertsekas2017}.

In general, it is challenging to guarantee feasibility in DP and ADP. 
For the special case of optimal control, DP schemes with feasibility guarantees have been developed
\cite{Bellman1957,Bjornberg2006,Mayne2001,Angeli2008,Bertsekas2017,Rakovic2006,Bjornberg2006, Jones2010}.
These approaches are, however, limited to  problems with convex polyhedral constraints defined over path graphs.
\Ae{Moreover, they are not directly transferable to other domains such as power systems
and they often neglect cost information.}
\Ae{In power systems, DP schemes are frequently discussed  \Ae{for } system operation under uncertainty \cite{Pacaud2022,Papavasiliou2018,Pinto2013}, where  \Ae{subproblems stem from uncertainty realizations or scenarios.} }
Moreover,  \cite{Zhu2022} develops an ADP-like scheme for  Optimal Power Flow~(OPF), where value functions are approximated recursively.
\Ae{Similarly, dual dynamic programming schemes and Bender's decomposition \cite{Flamm2018, Lohndorf2013,Philpott2008,Alguacil2000} construct value functions iteratively. 
Combining these  approaches with  feasibility cuts, which approximate the feasible set of the value functions iteratively leads to asymptotic feasibility guarantees \cite{Fullner2021}.
Doing so, however, these schemes lose their one-shot feasibility property.}

In a parallel  branch of power systems research, methods for \textit{flexibility aggregation} based on set projections  have been developed.
\Ae{These approaches mainly focus on the computation of a feasible, yet suboptimal solution in one shot.}
Similar to DP,  they require tree-structured problems such as OPF \cite{MayorgaGonzalez2018,capitanescu2018, silva2018,contreras2018,nazir2022,givisiez2020,Tan2019,Zhao2016,Muller2019a, Kellerer2015,Capitanescu2023,Tan2023a} or  charging of electric vehicle  \cite{Evans2022,Angeli2021,Ozturk2022,Appino2021}.
These approaches recursively approximate the feasible set of tree-structured problems starting at the leaves of the tree.
\TF{However, in the literature the close connection to DP/ADP is often overlooked.}
\Ae{Hence, cost information is typically neglected and
formal feasibility guarantees are not provided.}


\TF{In the present work, we generalize the above approaches via a unifying framework based on ADP.
We consider rather generic tree-structured optimization problems including constraints (\autoref{sec:probForm}), and thereby  go beyond the previous approaches  \cite{Bjornberg2006,Mayne2001,Rakovic2006,Angeli2008, Jones2010}.}
\Ae{At the same time, our approach is able to consider cost information either exactly or approximately, and it is guaranteed to deliver a feasible solution in one shot.}
\Aee{To the best of our knowledge, the combination of the above properties has so far not been documented in the literature.}

\Ae{\autoref{sec:DP} recalls elements of classic DP. }
\Ae{In \autoref{sec:FP-ADP},} we \Ae{develop the proposed ADP scheme and prove} feasibility guarantees  for general constraint sets. 
The conceptual idea is to consider constraint \Ae{and cost} information of lower-level problems via \Ae{extended real-valued value functions and} set projections, or inner approximations thereof.
These projections can---analogously to value functions in standard DP---be included in the \Ae{parent vertices}  to ensure feasibility \Ae{in the children.}
To discuss the potential for implementation synergies with existing numerical frameworks, 
\TF{\autoref{sec:comp} discusses how}  tools from computational geometry enable \Ae{the computation of set projections for convex polyhedral constraints} \cite{Jones2005,Herceg2013,Jones2010,Bjornberg2006,Angeli2008}.
For nonconvex problems, we propose \TF{a tailored variant of our scheme,} which ensures feasibility \Ae{via} inner approximations of the projection.
Moreover, we discuss the prospect of design centering \cite{Harwood2017} and semi-infinite optimization \cite{Stein2012,Djelassi2021} for obtaining inner approximations in \Ae{\autoref{sec:comp}.}
In order to simplify  computation, the proposed approach can be combined with value function approximations.
 \Ae{Finally, \autoref{sec:caseStudies} demonstrates the applicability of the proposed ADP scheme when combined with appropriate computational tools via examples from power systems and  optimal control.
 \autoref{sec:conclusions} concludes this work. }

\Ae{\textit{Notation:} 
	Given a set of vectors $\{x_i\}_{i \in \mathcal S}$, $\operatorname{vcat}(\{x_i\}_{i \in \mathcal S})=[x_1^\top,\dots,x_{|\mathcal S|}^\top]^\top$ denotes their vertical concatenation in ascending order.
	\TF{We call $\{\mathcal Z_i\}_{i \in S}$ with $\mathcal S = \{1,\dots,S\}$ an \emph{$S$-partition of $\mathcal Z$}, if $\mathcal Z_i \subseteq \mathcal Z$ for all $i \in \mathcal S$, $\cup_{i \in \mathcal S} \mathcal{Z}_i = \mathcal Z$ and all $\mathcal Z_i$ are mutually disjoint.}
	$\mathbb R_{\infty}\doteq\mathbb R \cup \{\infty\}$.
		\Ae{For basic  graph notions used in the context of this work we refer to  \cite{Vandenberghe2015,Bullo2018}.}
}

\section{Tree-Structured Optimization Problems} \label{sec:probForm}
We consider  \Ae{ problems} of the form
\begin{align} \label{eq:sepNLP1}
	\min_{x\in \Ae{\mathbb{R}^{n_x}}} f(x) \doteq \min_{x\in \Ae{\mathbb{R}^{n_x}}} \sum_{i \in \mathcal S} f_i(x_{\Ae{\mathcal I_i}}),
\end{align}
\TF{which are partitioned into subsystems $\mathcal S=\{1,\dots,|\mathcal S|\}$. }
\Ae{Each subsystem $i\in \mathcal S$ is endowed with an objective function $f_i:\mathbb R^{|\mathcal I_i|} \rightarrow \mathbb R_{\infty}$, which depends  on 
$x_{\mathcal I_i} \in  \mathbb{R}^{|\mathcal I_i|}$.  
Here,  $x_{\mathcal I} \in \mathbb R^{|\mathcal I|}$ denotes the components of $x$ selected by the index set $ \mathcal I \subseteq \{1,\dots,n\}$ in ascending order, cf. \Ae{ \cite{Vandenberghe2015,Bertele1972}.  We refer to $x_{\mathcal I} \in \mathbb R^{|\mathcal I|}$ as a \Ae{subvector}  of the global decision vector $x \in \mathbb{R}^{n_x}$. }
Note that $f_i$ is extended real-valued, i.e., $f_i$ considers subproblem constraints implicitly  via its domain.}

\Ae{Next, we analyze the  inherent graph structure of problem~\eqref{eq:sepNLP1}, which we would like to exploit via dynamic programming.
The following notions are based on \Ae{ \cite[Chap. 7]{Vandenberghe2015},  \cite{Bertele1972,Rosenthal1982}.}
We define the \emph{interaction graph} 
\begin{align*}
G = (\mathcal S, \mathcal{E}), \;\; \text{ where } \;\;  	\mathcal E \doteq \{ (i,j) \in \mathcal S\times \mathcal S \; |\; \mathcal I_i \cap \mathcal{I}_j \neq \emptyset\}.
\end{align*}
}
\Ae{Loosely speaking,  the tuple $(i,j)$ is in the edge set $\mathcal E$ if the corresponding subsystems share at least one component of the decision  vector $x$.}
Moreover, we define the set of \emph{coupling variables} between two subsystems $i$ and $j$ by $\mathcal W_{ij} = \mathcal W_{ji}  \doteq \mathcal I_i \cap \mathcal I_j$.
The set of \emph{local variables} is given by $\mathcal L_i \doteq \{ k \in \mathcal I_i \;|\; k \notin \mathcal I_j \text{ for all } i\neq j \}$.
\begin{example}[Interaction graph] \label{ex:coup}
	Consider the problem 
	\begin{align*}
		\min_x \; f_1(x_1,x_2) + f_2(x_2,x_3,x_4) + f_3(x_1,x_4),
	\end{align*}
	\TF{which is in in form of \eqref{eq:sepNLP1}.}
	Local and coupling variables are
\begin{align*}
				\mathcal I_1 &= \{1,2\},&\mathcal I_2 &= \{2,3,4\},& \mathcal I_3 &= \{1,4\},\\
		\mathcal L_1  &= \emptyset, & \mathcal L_2 &= \{3\}, & \mathcal L_4 &= \emptyset, \\
		\mathcal{W}_{12} &= \{2\},& \mathcal{W}_{23} &= \{4\}, & \mathcal{W}_{13} &= \{1\}.
	\end{align*}
	\TF{Observe that since $\mathcal E = \{(1,2),(2,3),(3,4)\}$, the interaction graph $G$ \Aee{contains a cycle of length three} and  is thus \emph{not} a tree \autoref{fig:coupvars} (left).}
	Removing the dependence of $f_3$ on $x_1$, e.g., would render the problem a tree, cf.  \autoref{fig:coupvars} (right).
	\hfill $\square$
\end{example}

\begin{remark}[Alternative graph encoding techniques]
	Note that problem \eqref{eq:sepNLP1} can equivalently be formulated via projection matrices $P_i$, i.e. $f_i(x_{\mathcal I_i}) = f_i(P_ix)$ with appropriately chosen $P_i$ \cite{Kim2009,Pakazad2017}, \cite[Chap. 12]{Vandenberghe2015}, \cite[Chap 7.4]{Nocedal2006}.
	Further approaches  define the graph structure  via the sparsity pattern of constraint matrices \cite{Shin2020,Engelmann2020b}. 
	Other works  express graph coupling in \eqref{eq:sepNLP1} via pairwise-affine coupling constraints \cite{Jiang2021} or they rely on hypergraphs \cite{Jalving2022}.
\hfill $\square$
\end{remark}

\Ae{
We conclude our problem exposition recalling a standard assumption of dynamic programming  tailored to our needs in terms of subsystem indexing:
\begin{assumption}\label{ass:tree}
	Problem \eqref{eq:sepNLP1} is feasible and $G$ is a tree. Moreover,  $1 \in \mathcal S$  is the root of the tree.	\hfill $\square$
\end{assumption}
Since $G$ is a tree, each vertex has one unique parent and  we  define $ \mathcal W_j \doteq \mathcal W_{\operatorname{par}(j),j}$.}

\begin{figure}
	\centering
	\includegraphics[width=0.45\linewidth]{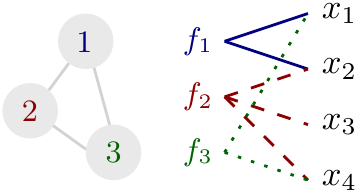}\hfill
	\includegraphics[width=0.45\linewidth]{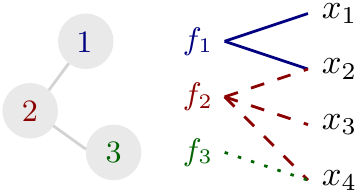}
	\caption{\Ae{Example for a general (left) and a tree-structured (right) interaction graph.}}
	\label{fig:coupvars}
\end{figure}

\section{Recap -- Dynamic Programming}
\label{sec:DP}

\TF{
Denote the children of subsystem $i \in \mathcal{S}$ in $G$ by $\mathcal C_i$.
Since \autoref{ass:tree} holds, we reformulate \eqref{eq:sepNLP1} via value functions \cite{Jiang2021} and this way generalize  DP to tree-structured problems. 
}
\begin{lemma}[Value function formulation of problem \eqref{eq:sepNLP1}] \label{lem:treeRef}
		Let Assumption~\ref{ass:tree} hold.
		Then, problem \eqref{eq:sepNLP1} can  be formulated 
 as		
			\begin{align} \label{eq:ValFunRef0} 
		V_i(x_{\mathcal W_i}) &\doteq \hspace{-.3cm}	\min_{\substack{x_{\mathcal L_i}, \\ \{x_{\mathcal W_j}\}_{j \in \mathcal C_i}}}  \hspace{-.3cm} f_i (x_{\mathcal L_i},\{x_{\mathcal W_i}\}_{j \in \mathcal C_i \cup \{i\}} ) \hspace{-.5mm}+\hspace{-.8mm} \sum_{j \in \mathcal{C}_i} V_j(x_{\mathcal W_j})
			\end{align}
	for all $i\in \mathcal S$. \hfill $\square$
	\end{lemma}
\Ae{A proof of this  observation is given in Appendix~\ref{sec:prelim}.}
Observe that in \eqref{eq:ValFunRef0} the dimension of the subproblems $V_i$ is reduced to the local variables $x_{\mathcal L_i}$ and  the coupling variables  $\{x_{\mathcal W_i}\}_{j \in \mathcal C_i}$ to all children $\mathcal C_i$ of subsystem $i \in \mathcal S$. 
Recall that  if $i$ is a leaf, $\mathcal C_i = \emptyset$ and that the root $i=1$ has no parent. Thus, $V_1$ does not depend on a $x_{\mathcal W_1}$.
It is worth to be noted that  \eqref{eq:ValFunRef0} is a generalization of the \Ae{finite-dimensional, deterministic} \emph{Bellman equation} 
to tree-structured problems, cf. \cite{Bertsekas2017,Powell2011}. 

\begin{algorithm}[t]
	\SetAlgoLined
	\caption{Classic Serial DP} \label{alg:classDP}
	\BlankLine
	\textbf{Initialize} $\mathcal H \doteq \mathcal L$ \\
	Backward sweep:
	\While{\textnormal{$\mathcal H \neq  \{1\}$}}{ \vspace{.27em}
		1) Compute  $V_j$ for all $j \in \mathcal H$ and send them to parent vertices  $ i\in \operatorname{par}(\mathcal H)$. \\
		2) Include  all $V_j$  in $V_i$ (eq. \eqref{eq:ValFunRef0}) of all $ i\in \operatorname{par}(\mathcal H)$. \\
				\phantom{$\mathcal H$} \\[.11em]
		3) Set $\mathcal H \leftarrow \operatorname{par}(\mathcal H)$.}
	Forward sweep: 
	\While{\textnormal{$\mathcal H \neq \mathcal L$}}{
		1) Solve  $V_j$ for all $j \in \mathcal H$. \\
		2) Distribute optimal coupling variables $\Ae{x_{\mathcal W_j}^\star}$ to all children $j \in \mathcal C_i$\\
		3) Set $\mathcal H \leftarrow \operatorname{chld}(\mathcal H)$.}
	\textbf{Return} \Ae{$x_{\mathcal I_i}^\star$ for all $i\in \mathcal S$}.
	\BlankLine
\end{algorithm}

\begin{figure}
	\centering
	\includegraphics[width=\linewidth]{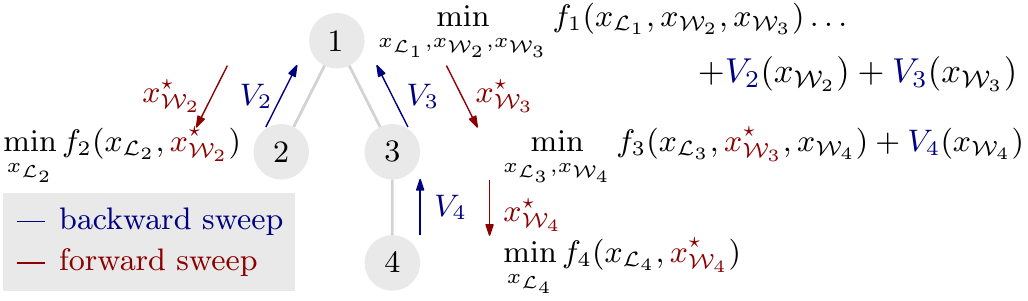}
	\caption{Classic serial DP example.}
	\label{fig:dpsweep}
\end{figure}

\begin{algorithm}[t]
	\SetAlgoLined
	\caption{Feasibility-Preserving ADP} \label{alg:appDP}
	\BlankLine
	\textbf{Initialize} $\mathcal H \doteq \mathcal L$ \\
	Backward sweep: \While{\textnormal{$\mathcal H \neq \{1\}$}}{
		1) Compute  $\mathcal P_i$ via \eqref{eq:proj}, and  $\tilde  V_i \approx \bar V_i$ for all $i \in \mathcal H$ and send them to parent vertices $ j\in \operatorname{par}(\mathcal H)$. \\
		2) Include all  $(\tilde V_i,\mathcal P_i)$ in $\bar V_j$ (eq. \eqref{eq:ValFunRef1}) of parent vertexs  $ j\in \operatorname{par}(\mathcal H)$ and compute  $\bar {\mathcal X}_j$. \\
		3) Set $\mathcal H \leftarrow \operatorname{par}(\mathcal H)$.}
	Forward sweep: \While{\textnormal{$\mathcal H \neq \mathcal L$}}{
		1) Solve  $\bar  V_j$ for all $j \in \mathcal H$. \\
		2) Distribute optimal coupling variables $\Ae{\bar x_{\mathcal W_j}^\star}$ to all children $i \in \mathcal C_j$\\
		3) Set $\mathcal H \leftarrow \operatorname{chld}(\mathcal H)$.}
	\textbf{Return} \Ae{$\bar x_{\mathcal I_i}^\star$ for all $i\in \mathcal S$}.
	\BlankLine
\end{algorithm}
\begin{figure}[t]
	\centering
	\includegraphics[width=\linewidth]{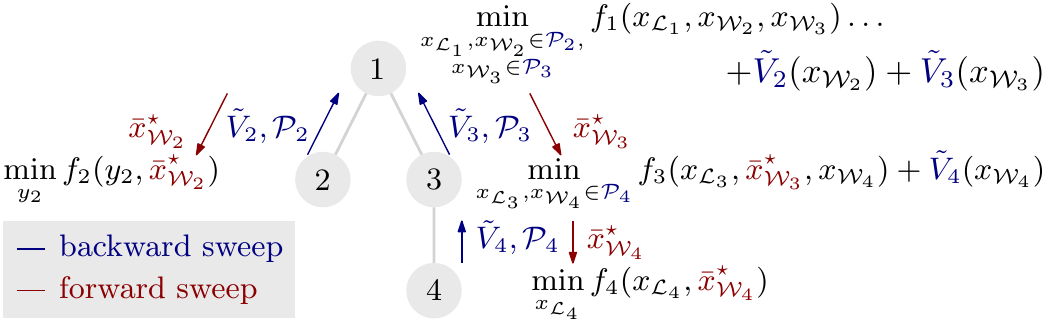}
	\caption{Feasibility-preserving ADP example. }
	\label{fig:FPdpsweep}
\end{figure}

\subsection{Classic (nonserial) Dynamic Programming } \label{sec:classDP}
Next, we recall classic \Ae{(nonserial)} dynamic programming \Ae{ for problems without constraints, which} solves~\eqref{eq:ValFunRef0} in one backward and one forward sweep along the tree:
In the backward sweep, start at leaves $l \in \mathcal L$ and compute an expression for all $V_l$.
Communicate all  $V_l$ to parent vertices, which include them to their objective, cf. problem~\eqref{eq:ValFunRef0}. 
This step is repeated recursively until the root vertex is reached. 
In the forward sweep, start at the root vertex, solve~\eqref{eq:ValFunRef0}, and communicate $\{\Ae{x_{\mathcal W_j}^\star}\}$ to all children $j \in \mathcal C_1$.
At all children $j \in \mathcal C_1$, solve~\eqref{eq:ValFunRef0} for $\{\Ae{x_{\mathcal W_j}^\star}\}$ fixed and repeat until all leaves $l \in \mathcal L$ are reached.
Classic dynamic programming is summarized in \autoref{alg:classDP}.
\Ae{\autoref{fig:dpsweep} illustrates the DP algorithm for \autoref{ex:coup} (right). }

Note that proving the optimality of the result of  \autoref{alg:classDP} \Ae{directly follows from the equivalence of \eqref{eq:sepNLP1} and \eqref{eq:ValFunRef0}.}
By this equivalence, \autoref{alg:classDP} can be interpreted as solving  \eqref{eq:sepNLP1} in a particular order starting from the  leaves.

\begin{remark}[Serial and nonserial dynamic programming]
	We remark that classic dynamic programming was originally designed  for path graphs, i.e. for serial decision processes without branching stemming from optimal control of dynamic systems \cite{Bellman1957,Bertsekas2017}.
	These approaches have been extended to more general tree structures in the 1960s and 1970s~ \cite{Bertele1972,Esogbue1974}.
	They have been  revisited recently in the context of decomposition methods for tree-structured  problems \cite{Jiang2021}, \cite[Chap 7.1]{Vandenberghe2015}. \hfill $\square$
\end{remark}

\subsection{Evaluation and communication of $V_i$} \label{rem:propV}
	Observe that evaluating $V_i$ is    costly in general since it involves solving an optimization problem.
	\TF{If it is possible to derive an explicit expression for $V_i$,}  it is in general difficult to handle in the context of numerical computation since it is a piecewise-continuous, nonlinear, nonconvex, extended real-valued function in general~\cite{Fiacco1990}.
	Except for special cases such as Quadratic Programs (QPs) or Linear Programs (LPs) \cite{Bemporad2000,Herceg2013,Keshavarz2014,Narciso2022}, however, it is  \TF{often intractable to compute such expressions.}
	One remedy are smoothing techniques via log-barrier functions  rendering $V_i$ differentiable   \cite{DeMiguel2008,Yoshio2021}.
	This way, one can  at least compute sensitivities of $V_i$, which in turn can be used for solving the parent problem.
\Ae{Eventually, this leads to \emph{primal decomposition} algorithms, cf. \cite{Tu2021,Engelmann2022b}.}

	Approximate Dynamic Programming~(ADP) schemes alleviate the above bottleneck by replacing  $V_i$ with suitable function approximators $ V_i \approx \tilde V_i$ such as neural networks or quadratic fits \cite{Bertsekas2017,Powell2011,Jiang2021}.
	\Ae{Piecewise-constant \cite{silva2018} and piecewise-linear \cite{Capitanescu2023} approximations based on samples are also possible.}
	However, ensuring feasibility with these techniques seems to be difficult.	
	\TF{As mentioned in the introduction, some DP schemes construct $\tilde V$ and its domain iteratively, e.g. in SDDP or Bender's decomposition via optimality and feasibility cuts \cite{Flamm2018, Lohndorf2013,Philpott2008,Alguacil2000,Fullner2021}. 
	However, these schemes do not have  one-shot feasibility properties.}

\section{Feasibility-preserving \\ Approximate Dynamic Programming} \label{sec:FP-ADP}

Ensuring feasibility when using valuefunction approximations seems to be difficult.
Next, we design a Feasibility-Preserving Approximate Dynamic Programming~(FP-ADP) scheme, which ensures feasibility and approximate optimality in one forward-backward sweep.
\Ae{For doing so, we use  concepts from variational analysis  \cite{Rockafellar1998a,Mordukhovich2005}. }

\subsection{Considering Constraints}
\Ae{Considering constraints $x_{\mathcal I_i} \in \mathcal X_i$ in  problem \eqref{eq:sepNLP1} implies that all $f_i$ are extended real-valued on $\mathbb R ^{n_{xi}}$.
	In this case,  constraint sets $\mathcal X_i$ are considered via indicator functions, i.e. $f_i \doteq l_i + \iota_{\mathcal X_i}$, where $l_i$ is the objective function without constraints and 
	\begin{align*}
		\iota_{\mathcal X_i}(x_{\mathcal I_i}) \doteq
		\begin{cases}
			0 &\text{ if }x_{\mathcal I_i} \in \mathcal X_i \\
			\infty &\text{ else}
		\end{cases}
	\end{align*}
	is the \emph{indicator function}.
\TF{Thus, the issue of determining whether a point $x_{\mathcal I_i}$ is feasible can equivalently be considered as the question of whether the condition  $f_i(x_{\mathcal I_i}) < \infty$ holds, i.e. whether $x_{\mathcal I_i}$ is in the  \emph{domain} $\operatorname{dom}(f_i) \doteq \{x_{\mathcal I_i} \in \mathbb{R}^{n_{xi}} \;|\; f_i(x_{\mathcal I_i}) < \infty\}$, of $f_i$,  cf. \cite{Rockafellar1998a}.}

In order to guarantee feasibility in \autoref{alg:classDP}, the key idea is to express and communicate $\operatorname{dom}(V_i)$  in \eqref{eq:ValFunRef0} to the parent subsystem $j=\operatorname{par}(i)$ such that the parent problem can include this  explicitly as a constraint in its own $V_j$.
However, it is not immediately clear how to compute $\operatorname{dom}(V_i)$.
The next lemma shows that $\operatorname{dom}(V_i)$ can be computed via  set projections.
\Ae{To this end, recall the} \emph{orthogonal projection} of a set $\mathcal X \subseteq \mathcal Z \times \mathcal Y \subseteq \mathbb{R}^n \times \mathbb R^m$ on $\mathcal Z$ \cite{Rakovic2006} given by
\begin{align*}
	\operatorname{proj}_{\mathcal Z}(\mathcal X) \doteq \{ z \in \mathcal Z  \;|\;\exists \; y \in \mathcal Y \text{ with  } (z,y) \in \mathcal X\}.
\end{align*}
Similarly,  we define the  orthogonal projection  of a set $\mathcal X  \subseteq \mathbb{R}^n$ on the components indexed by  $\mathcal I \subseteq \{1,\dots,n\} $ as
\begin{align*}
	\operatorname{proj}_{\mathcal I}(\mathcal X) \doteq \left \{ x_{\mathcal I} \subseteq \mathbb{R}^{|\mathcal I|} \;|\;\exists \; x_{\{1,\dots,n\} \setminus \mathcal I} \text{ with  } x \in \mathcal X \right \} .
\end{align*}
}

\Ae{\begin{lemma}[Computing domains via projection] \label{lem:domVi}
	\Ae{	Consider  
			\begin{align} 
 V(z) &\doteq \min_y F(y,z),  \label{eq:valFunProb3}
			\end{align}
		where   $ F: \mathbb{R}^{n_y}\times \mathbb{R}^{n_z} \rightarrow \mathbb{R}_{\infty}$, and assume that the minimum exists.\footnote{\Ae{In order to avoid technicalities, we assume the existence of minimizers at several occasions. This assumption can be replaced by appropriate continuity requirements on $f_i$ and compactness properties on $\operatorname{dom}(f_i)$, cf. \cite[Chap 1.C]{Rockafellar1998a}.}}
		Then,  
		$	\operatorname{dom}(V) = \operatorname{proj}_{\mathcal Z}{(\operatorname{dom}(F))}$,
		where $\operatorname{dom}(F) \doteq \{(y,z)\;|\; F(y,z) < \infty\}$. \hfill $\square$
\end{lemma}}
\begin{proof}
\Ae{See Appendix~\ref{sec:prelim}.}
	\end{proof}
}

\Ae{The above lemma allows to rewrite minimization problems in form of \eqref{eq:valFunProb3} as
	\begin{align} \label{eq:projProb2}
		\min_z \; V(z) = \min_{z \in \mathcal P} V(z), 
	\end{align} 
	where
	 \begin{align} \label{eq:Pi}
	 \mathcal P \doteq \operatorname{proj}_{\mathcal Z}{(\operatorname{dom}(F))}  = \operatorname{dom}(V)
	\end{align}
	Moreover, it allows to formulate feasibility-preserving approximations of \eqref{eq:valFunProb3},
	\begin{align*} 
		\min_{z \in  \mathcal P} \tilde V(z),
	\end{align*} 
	where $\tilde V \approx V$.
	Furthermore, one can combine the above with inner approximations $\tilde {\mathcal P} \subseteq \mathcal P$	yielding
		\begin{align} \label{eq:valFunProb4}
		\min_z \; V(z) \approx\min_{z \in \tilde {\mathcal P}} \tilde V(z),
		\end{align} 
	 the solution of which---if it exists---is again feasible in terms of problem \eqref{eq:valFunProb3}.
	This implies that we can combine the reformulation \eqref{eq:valFunProb3} with a very general class of  real-valued approximations $\tilde V$ and inner approximations $\tilde {\mathcal{P}}$, while guaranteeing feasibility in terms of \eqref{eq:valFunProb3}.
		We summarize these observations in the following theorem.}
	\begin{theorem}[Constraint consideration via projections]  \label{thm:feasAndOpt}
	\Ae{Consider $V$ from  \eqref{eq:valFunProb3}, the problems \eqref{eq:projProb2} and \eqref{eq:valFunProb4}
	with $\mathcal P$ from \eqref{eq:Pi}, and some $\tilde V: \mathcal P \rightarrow \mathbb R$.
	Assume that the minimizers  in  \eqref{eq:projProb2}--\eqref{eq:valFunProb4} exist.
	Then,  $z^\star$ is a minimizer to   \eqref{eq:valFunProb3}  if and only if it is a minimizer to problem \eqref{eq:projProb2}.
	Moreover, any minimizer $\tilde z^\star$ to problem \eqref{eq:valFunProb4} is feasible in  \eqref{eq:valFunProb3}, i.e. $V(\tilde z^\star) < \infty$.} \hfill $\square$
	\end{theorem}
	\begin{proof}
		Since all minimizers exist, we can use $\inf$ and $\min$ interchangeably.
		The first statement  follows from  $\inf_x f(x) = \inf_{x \in \operatorname{dom}(f)}f(x)$ \cite[Chap 1.A]{Rockafellar1998a} and \autoref{lem:domVi}. 
		For the second statement, note that by assumption a minimizer $\tilde x^\star$ exists and thus  $\tilde x^\star \in \mathcal P$.
		Since $\tilde {\mathcal P} \subseteq \mathcal P$, we have $\tilde x^\star \in \operatorname{dom}(V)$.
		\end{proof}
\subsection{Feasibility-Preserving ADP}

We use the above result to design a Feasibility Preserving ADP scheme, which ensures feasibility although using value function approximations.
\Aee{
To match the requirements of \autoref{lem:domVi}, we introduce 
\begin{align} \label{eq:Fi}
	  F_i\big (x_{\mathcal L_i},\{x_{\mathcal W_i}&\}_{j \in \mathcal C_i \cup \{i\}}\big ) \\
&\quad \doteq
  \;f_i\left (x_{\mathcal L_i},\{x_{\mathcal W_i}\}_{j \in \mathcal C_i \cup \{i\}}\right ) +
  \sum_{j \in \mathcal{C}_i} V_j(x_{\mathcal W_j}). \notag
\end{align}
Observe that $F_i$ is the objective of problem \eqref{eq:ValFunRef0}, i.e. $V_i(x_{\mathcal W_i}) =	\min_{x_{\mathcal L_i},  \{x_{\mathcal W_j}\}_{j \in \mathcal C_i}} F_i (x_{\mathcal L_i},\{x_{\mathcal W_i}\}_{j \in \mathcal C_i \cup \{i\}} )$.}

In the backward sweep, start at leaves $i \in \mathcal L $. 
\Ae{Recall that $\mathcal W_i$ is the index set of coupling variables of subsystem $i \in \mathcal S$ with its parent.}
\Aee{Using \autoref{lem:domVi}, we define}
\begin{align} \label{eq:proj}
	\mathcal P_i \doteq\Aee{\operatorname{dom}(V_i)=} \Ae{\operatorname{proj}_{\mathcal W_i}( \operatorname{dom}(\Aee{F_i}))},
\end{align}
\Ae{i.e. the projection of the constraint set $\mathcal X_i=\operatorname{dom}(F_i)$ of subsystem $i\in \mathcal S$ onto the coupling variables with its parent.}
\Ae{The central idea of FP-ADP is to}  combine the feasible sets of all  \TF{children  $i\in \mathcal C_i$  via the projections $\mathcal P_i$} or approximations thereof. 
\Ae{To this end, we approximate problem \eqref{eq:ValFunRef0} via}
\begin{align}\label{eq:ValFunRef1}
 \bar V_j(x_{\mathcal W_j})  \doteq &\min_{\substack{x_{\mathcal L_j},  \\ \{x_{\mathcal W_i}\}_{i \in \mathcal C_j}}}   f_j(x_{\mathcal L_j}, \{x_{\mathcal W_i}\}_{i \in \mathcal C_j \cup \{j\}}) +  \sum_{i \in \mathcal{C}_j} \tilde V_i(x_{\mathcal W_i}) \notag\\
& \text{subject to }\quad \Ae{ x_{\mathcal W_i} \in \mathcal P_i \text{ for all } i\in \mathcal C_j},
\end{align}
\Ae{where $\tilde V_i: \mathcal P_i \rightarrow \mathbb R$ are real-valued approximations of $\bar V_i$, \Aee{and $\bar F_i$ is defined analogously to  \eqref{eq:Fi}.}
Observe the subtle differences between $V_i, \tilde V_i,$ and $\bar V_i$: 
$V_i$ is the exact value function for all subsystems as defined in~\eqref{eq:ValFunRef0}.
$\bar V_i$ describes the {exact} value function   including approximations in the children $\tilde V_k, k \in \mathcal C_i$, and $\tilde V_i$ again approximates $\bar V_i$ in the parent subsystem $i \in \mathcal S$.}

Similar to classic dynamic programming from \autoref{sec:classDP}, we recursively compute and communicate $(\tilde V_j,\mathcal P_j)$ backwards until the root vertex is reached.
Observe that the root problem yields a set \Ae{$ {\mathcal P_1}$}, which ensures feasibility of all lower-level problems.
In the forward sweep, start at the root, solve $\bar V_1$ and communicate \Ae{ $\{\bar x_{{\mathcal W}_j}^\star\}$} to all children $j \in \mathcal C_1$.
At all children $j \in \mathcal C_1$, solve~$\bar V_j$ and repeat until all leaves $l \in \mathcal L$ are reached.
The FP-ADP scheme is summarized in \autoref{alg:appDP}.
\autoref{fig:FPdpsweep} illustrates the algorithm for \autoref{ex:coup}, where the edge \Ae{$\{1,3\}$} has been removed from the interaction graph.

\begin{theorem}[Feasibility in  \autoref{alg:appDP}] \label{thm:feasADP}
	Suppose that  the minimizers of all $\{\bar V_j\}_{j \in \mathcal S}$ in \autoref{alg:appDP} exist.
Then, \autoref{alg:appDP}  generates a feasible point to problem \eqref{eq:sepNLP1}.\hfill $\square$
\end{theorem} 
\begin{proof}
Since \Ae{the minimizers of all $\{\bar V_j\}_{j \in \mathcal S}$ exist,}  $\Ae{\operatorname{dom}(\bar V_l)} \neq \emptyset$  for $l \in \mathcal L$. 
By  definition of the set projection \Ae{and \eqref{eq:proj}}, this implies that  $\Ae{\mathcal P_j}\neq \emptyset$ for all $j = \operatorname{par}(l)$.
Induction gives that   $\Ae{\mathcal P_j}\neq \emptyset$ for all $j \in \mathcal S$. 
Hence, $\Ae{\bar x_{\mathcal W_j}^\star }\in \arg \min \bar V_j(z_j)$ exists for all $j \in \mathcal S$ by assumption and\Ae{, thus, in particular the minimizer of $\bar V_1$ exists.  \Ae{The definition of the set projection implies that $\bar x ^\star_{\mathcal I} \in \operatorname{dom}(f)$.}}
\end{proof}
\autoref{thm:feasADP} can be  extended to the case, where $\mathcal P_j$ is replaced by a (possibly more tractable) inner approximation
$
	\tilde{\mathcal P_j} \subseteq \mathcal P_j.
$
\begin{corollary}[Feasibility  with inner  approximation] \label{thm:feasADPappr}
	Assume that the assumptions from \autoref{thm:feasADP} hold.
Replace $\mathcal P_j$ by inner approximations $\mathcal {\tilde P}_j \subseteq \mathcal P_j$ such that $\Ae{\operatorname{dom}(\bar V_i) \neq \emptyset}$ for all $i \in \mathcal S$. 
Then, \autoref{alg:appDP} is well-defined and generates a feasible point  $\bar x^\star_{\mathcal I} \in \Ae{\operatorname{dom}(f)}$. \hfill $\square$
\end{corollary}

\section{Computational Aspects} \label{sec:comp}

In the following section, we describe ways of \Ae{computing and parametrizing} $\mathcal P_i$ and $\tilde {\mathcal P}_i$ such that they can be used efficiently in parent optimization problems.
\Ae{We will use these computations in the numcerical examples of \autoref{sec:caseStudies}.
	This way, we show that the described FP-ADP schemes are indeed implementable in practice.}
We do not consider the computation of value function approximations $\tilde V_i$, cf.  \autoref{rem:propV} for comments.
\Ae{For  simplicity, we set $z_i\doteq x_{\mathcal{W}_i}$ and $y_i\doteq (x_{\mathcal{L}_i}\Aee{,\{x_{\mathcal W_j}\}_{i\in \mathcal C_i})}$, and we drop the subscript $i$ in the following.}

\subsection{Computing Exact Projections } \label{sec:Pexact}
If child problems have convex polyhedral constraint sets
\begin{subequations}\label{eq:polCstr}
	\begin{align} 
 \Aee{\operatorname{dom}(F)} \doteq \big \{(z,y) \;|\; [A_z \;\; A_y][z^\top\;\; y^\top ]^\top&=b, \\
[B_z \;\; B_y][z^\top\;\; y^\top ]^\top&\leq d\big \},
\end{align}
\end{subequations}
\Ae{then $\mathcal P = \operatorname{proj}_{ \mathcal Z}{\Ae{(\operatorname{dom}(F))}}$ can be  computed \emph{exactly}.}
In this case, $\mathcal P$ is a convex polyhedron of the form $\mathcal P=\{ z \;|\; \bar A_z z = 0, \;\; \bar B_z z\leq 0\}$, which can be efficiently included in and communicated to the parent optimization problem.
Algorithms such as Fourier-Motzkin elimination or vertex enumeration can be used to compute $\mathcal P$.
Depending on the dimensions of $y$ and $z$, these algorithms vary with respect to their efficiencies.
We refer to \cite{Jones2005,Jones2008} and references therein for an overview, and to \cite{benoit2023,Fukuda1996} for implementations.

\subsection{Computing Inner Approximations} \label{sec:PinnerAppox}
\TF{For  sets $\operatorname{dom}(F)$ with complex topology} it is often computationally intractable to obtain $\mathcal P$ exactly.
\Ae{However, as shown in Theorem~\autoref{thm:feasADPappr}, non-empty and not overly conservative inner approximation  $\tilde {\mathcal P} \subseteq \mathcal P$ suffice to ensure feasibility in FP-ADP. } 
Moreover,  using an inner approximation can help to simplify the representation of $\mathcal P$, e.g., if $\mathcal P$ is a convex polyhedron  described by a large number of halfspaces.

\subsubsection*{Sampling-Based Methods}
Consider feasible sets of NLPs 
\begin{align} \label{eq:NLPset}
  \Aee{\operatorname{dom}(F)}\doteq \{ (z,y) \in \mathbb R^{n_x} \;|\; g(z,y)=0, \;h(z,y) \leq 0\}.
\end{align} 
\Ae{One possibility to approximate $\mathcal P = \operatorname{proj}_{ \mathcal Z}{(\operatorname{dom}(F))}$ are sampling-based heuristics often used in power system applications \cite{Krause2009,capitanescu2018, silva2018,MayorgaGonzalez2018,Sarstedt2021}.
These methods   \Ae{sample $T\in \mathbb N$} points  $\{\mathsf z_l\}_{l \in \mathcal T=\{1,\dots,T\}} \in \mathcal P$ based on which $\tilde {\mathcal P}$ is constructed.}

\TF{
One of these methods starts with a set of random samples  $\{\mathsf y_l\}_{l \in \mathcal T}$ with $\mathsf y_l \in \mathcal Y$.
Then, $g(\mathsf z_l,\mathsf y_l)=0$ is solved numerically, which  yields a sequence $\{\mathsf z_l\}_{l \in \mathcal T}$.
Neglecting all $l \in \tilde  {\mathcal T} \subset \mathcal T$ for which $h(\mathsf z_l,\mathsf y_l) > 0$ holds,   guarantees a posteriori  that $\mathsf z_l\in \mathcal P$ for all $l \in \mathcal T \setminus \tilde {\mathcal T}$ \cite{MayorgaGonzalez2018, heleno2015}.}
\Ae{This approach is summarized in \autoref{alg:discMeth1}.}

More recent methods explore the boundary of $\mathcal P$ more systematically in order to avoid excessive sampling.
Boundary points $\{\mathsf z_l\}$ are computed by a sequence of auxiliary optimization problems $l \in \mathcal T$
\begin{align} \label{eq:boundaryProb}
\mathsf z_l = \arg \min_{z,y} c_l\, z\;  \text{ subject to } \; g(z,y)=0, \; h(z,y) \leq 0,
\end{align}
with varying costs coefficients typically chosen to be $c_l \in \{-1,0,1\}^{n_z}$ 
\cite{capitanescu2018,silva2018}.
These methods can be extended by using $\{\mathsf z_l\}_{l \in \mathcal T}$ \TF{as a base set of points and adding further ones via gridding.}
Specifically, assume that the $m$th component of $\mathsf z$, $[\mathsf z]_m$ is bounded by  $\underline  {\mathsf z}_l \leq [\mathsf z_l]_m \leq \bar {\mathsf z}_l$ for all $l \in \mathcal T$.
Then, one can fix these components for a finite set of gridding points in the above range and re-solving \eqref{eq:boundaryProb} for the corresponding components fixed in order to get a finer sampling of the boundary of $\mathcal P$, cf. \cite{Sarstedt2021} and references therein.
This  method is detailed in \autoref{alg:discMeth2}.

\begin{algorithm}[t] 
	\SetAlgoLined 
	\caption{\Aee{Decision Variable Sampling} } \label{alg:discM}
	\BlankLine
	\textbf{Initialize} \Aee{$\operatorname{dom}{(F)}$} in form of  \eqref{eq:NLPset}. \\
	1) Choose a set of samples $\{\mathsf y_l\}_{l \in \mathcal T}$ with $\mathsf y_l \in \mathcal Y$.  \\
	2) Solve $g(\mathsf z_l,\mathsf y_l)=0$ numerically to obtain $\{\mathsf z_l\}_{l \in \mathcal T}$.\\
	3) Neglect infeasible points  $\tilde {\mathcal T} =\{l \;|\; h(\mathsf z_l,\mathsf y_l)>0\}$. \\
	\textbf{Return} $\{\mathsf z_l\}_{l \in \mathcal T \setminus \tilde {\mathcal T}},\; \mathsf z_l \in \mathcal P $.
	\BlankLine \label{alg:discMeth1}
\end{algorithm}

Typically it is not desirable to use \Ae{the set of samples}  $\{\mathsf z_l\}_{l \in \mathcal T} $ as \Ae{constraints on the coupling variables}  in a parent optimization problem since that would lead to a mixed-integer problem.
Hence, the samples are often used to compute an  approximation $\tilde {\mathcal P}$, e.g.  $ \tilde {\mathcal P} = \operatorname{conv}(\{\mathsf z_l\})$, \Ae{where $\operatorname{conv}(\cdot)$ denotes the \Ae{convex hull.}}
Alternatives are fitting shapes such as ellipsoids or convex polyhedra with maximal volume \cite{polymeneas2016}.
However, without further assumptions on $\mathcal P$ such as convexity or at least simply connectedness, the approximations constructed via the above methods are, \Ae{in general,  not guaranteed to satisfy  $\tilde {\mathcal P} \subseteq \mathcal P$ for constraint sets of NLPs \eqref{eq:NLPset}.  In turn, this may} jeopardize feasibility in the forward sweep of Algorithm~\ref{alg:appDP}.

\begin{figure*}
	\centering
	\includegraphics[width=0.9\linewidth]{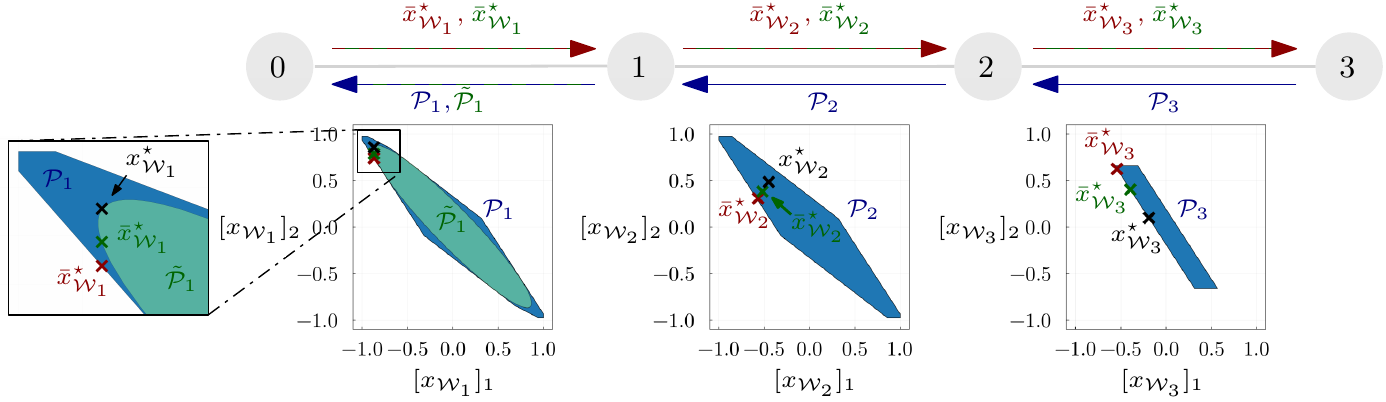}
	\caption{FP-ADP for OCP \eqref{eq:OCP}, where $\mathcal S$ are time instances and $G$ is a path graph.}
	\label{fig:ocpdp}
\end{figure*}

\subsubsection*{Constraint Simplification Methods}
Assume that one is able to compute a projection  described via inequalities  $\mathcal P = \{z\;|\; \tilde g(z) \leq 0\} \subset \mathcal P$, e.g. via the projection of polyhedra from \eqref{eq:polCstr}.
Suppose further that we would like to  simplify this description.
\Ae{Indeed,} methods from \emph{design centering} allow to compute inner approximation $\tilde {\mathcal P}\subseteq \mathcal  P$ \cite{Harwood2017,Stein2012,Djelassi2021,Jones2010,Bjornberg2006}.
\Ae{The pivotal idea 
is to choose an easily parametrizable shape} $\mathcal Q$ such as a polyhedron or an ellipsoid, and to set $\tilde {\mathcal P} \doteq \mathcal Q(v)$, where $v \in \mathcal V$ are design parameters.
Then, one tries to maximize the \textit{volume} of $\mathcal Q(v)$ with parameters from a design space $v \in \mathcal V$ such that $\mathcal Q(v) \subseteq \mathcal P$.
Thus, the design centering problem reads
\begin{align} \label{eq:desCtr}
\max_v \quad \operatorname{vol}(\mathcal Q(v) )	\text{ subject to }
\mathcal Q(v) \subseteq  \mathcal P, \; v \in \mathcal V.
\end{align}
As an example, assume that $\mathcal P$ is given by a convex polyhedron $\mathcal P=\{ z \;|\; \bar B_z z\leq 0\}$, for example from \eqref{eq:polCstr}.\footnote{Note that equality constraints in convex polyhedra can always be eliminated by a suitable coordinate transformation (a.k.a. the nullspace method) \cite[Chap. 16.2]{Nocedal2006}, \cite[Chap 10.1.2]{Boyd2004}.}
Moreover, assume that we would like to replace $\mathcal P$ by a simpler ellipsoidal inner approximation.
An ellipsoid can be described as a Euclidean unit ball under an affine mapping \cite[Chap 4.9]{Ben-Tal2001a}, i.e.
\begin{align*}
\mathcal Q^e(v) = \big \{ y = Au+c \;|\;   \|u\|_2^2 \leq 1,   \;A^\top \hspace{-.4em} =A \succ 0 \big \},
\end{align*}
where  $v= (A,c)$.
Then, we can solve \eqref{eq:desCtr} via the maximum volume inscribed ellipsoid problem
\begin{subequations} \label{eq:ellFitDet}
\begin{align} \label{eq:ellFit}
	\max_{(A= A^\top\succ 0,c)} \quad \operatorname{det} A&	\\
	\text{ s.t. } 	\big \|A[B_z]_i\big \|_2 + [B_z]_i c +  &\leq  [d]_i, \;  i = 1,\dots, N, \label{eq:robPolInq2} 
\end{align}
\end{subequations}
where $N$ is the number of rows of $B_z$.
Observe that \eqref{eq:ellFitDet} is a semidefinite program, cf. \cite[Chap 8.4.2]{Boyd2004} and \cite{Ben-Tal2001a}.
An application of this technique is illustrated in \autoref{sec:optCtrl}.

\Aee{
A set projection can equivalently be expressed as a linear mapping of a set, i.e. $\operatorname{proj}_{\mathcal Z}\mathcal X = P_{\mathcal Z}\mathcal X$ with $P\mathcal X \doteq \{Px\;|\; x\in \mathcal X\}$, where $P_{\mathcal Z}$ is an appropriately chosen projection matrix. Hence, the projection of an ellipsoid of the above form reads $P\mathcal Q^e=\{ y = PAu+c \;|\;   \|u\|_2^2 \leq 1,   \;A^\top \hspace{-.4em} =A \succ 0 \big \}$ and is thus reduces to a matrix multiplication.
The same holds true for zonotopes, which can be used as an alternative \cite{Scott2016}.}

\section{Examples} \label{sec:caseStudies}

\Ae{Next, we illustrate how to use the methods from the previous subsection in Algorithm~\ref{alg:appDP} for two  applications from optimal control and  power systems.} 

\TF{We use  \texttt{PowerModels.jl} for power system modelling \cite{Coffrin2018},  \texttt{Polyhedra.jl} \cite{benoit2023} for polyhedral projection, and  \texttt{Convex.jl} and \texttt{JuMP} with \texttt{Ipopt} \cite{Dunning2017,Wachter2006} for solving optimization problems.}

\begin{algorithm}[t] 
	\SetAlgoLined 
	\caption{\Aee{Optimization-based Sampling}} 
	\BlankLine
	\textbf{Initialize} \Aee{$\operatorname{dom}(F)$} in form of  \eqref{eq:NLPset}. \\
	1) Choose a set of sampling cost coefficients $\{c_l\}_{l \in \mathcal T}$. \\
	2) Solve \eqref{eq:boundaryProb} for all $l \in \mathcal T$. \\
	\textbf{Return} $\{\mathsf z_l\}_{l \in \mathcal T},\; \mathsf z_l\in \mathcal P$.
	\BlankLine \label{alg:discMeth2}
\end{algorithm}

\subsection{Optimal Control} \label{sec:optCtrl}
Consider a discrete-time Optimal Control Problem~(OCP)
\begin{subequations}\label{eq:OCP}
	\begin{align} 
		\min_{z,u} \sum_{t \in \mathbb T} &\ell(z_t,u_t) + V_T(z_t)\\
		\text{s.t. } \quad  		z_{t+1} &= f(z_t,u_t), \; z_0 = z(0),&& \forall t\in \mathbb T \setminus \{T\}, \label{eq:dyn}\\
		(z_t,u_t)  &\in \mathbb{Z} \times \mathbb{U},&&\forall t \in \mathbb  T \setminus \{T\}, \label{eq:StInCstr} \\
		z_t & \in \mathbb{Z}_T, \label{eq:termSet}
	\end{align}
\end{subequations}
where $\mathbb{T}\doteq\{0,\dots,T\}$.
OCP \eqref{eq:OCP} is in form of problem~\eqref{eq:sepNLP1} if we define the following:
 $\mathcal{S}\doteq\mathbb{T}$\footnote{\Ae{In contrast to the previous sections, we use  $0$ as the root of the interaction tree here in order to keep the established zero index for the initial condition.}}, $i\doteq t$, $\Ae{x_{\mathcal I_i}} \doteq (z_{i-1},z_i,u_i)$ for all $i\in \mathcal S \setminus \{0\}$, and  $\Ae{x_{\mathcal I_0}} \doteq (z_0,u_0)$.
\Ae{Furthermore, 
	$f_i(x_{\mathcal I_i}) \doteq \ell(z_i,u_i) + \iota_{\mathcal X_i}(x_{\mathcal I_i}),$
where
\begin{align*}
	 {\mathcal X}_i \doteq \left  \{x_{\mathcal I_i} \;|\; \eqref{eq:dyn} \text{ (left) and }  \eqref{eq:StInCstr} \text{ hold}\right\}
\end{align*}
for all $ i\in \{1,\dots,T-1\}$.
We consider $\mathcal X_0 \doteq \{x_{\mathcal I_0} \;|\; \eqref{eq:dyn} \text{ (right), } \eqref{eq:StInCstr}\} \text{ hold}\}$, and $\mathcal X_S \doteq \{x_{\mathcal I_S} \;|\; \eqref{eq:termSet}  \text{ holds} \}$.
}

The above OCP \Ae{is structured by} a path graph, since only two consecutive time steps  $(i,i+1)$ are linked via the system dynamics \eqref{eq:dyn} and all other constraints are separable in time.
\TF{Using the same rationale, we have $x_{\mathcal W_i} = z_{i-1}$ for all $i = 2,\dots,T$, i.e. , the coupling variable to the previous time step is its state.}
The resulting \Ae{path graph} is illustrated in \autoref{fig:ocpdp}.


\subsubsection*{Numerical Example}

Consider \eqref{eq:OCP} with quadratic stage cost $\ell(z,u) \doteq \|z\|_2^2  + 0.1\,\|u\|_2^2$ and  linear system dynamics
\begin{align} \label{eq:sysDyn}
	f(z_t,u_t) \doteq
	\begin{bmatrix}
		1.5 & 1 \\
		0 & 1.5
	\end{bmatrix}
	z_t
	+ 
	\begin{bmatrix}
		0 \\ 0.8
	\end{bmatrix}
	u_t.
\end{align}
Moreover, set $z(0) = [-0.3 \;\; -0.2]^\top$, $\mathbb U = \{- 1\leq u\leq 1\}$ and $T = 3$.
When using \eqref{eq:OCP} in the context of Model Predictive Control~(MPC), computing a feasible solution to \eqref{eq:OCP} guarantees closed-loop stability, if the terminal penalty $V_3$ and the terminal constraint set $\mathbb Z_3$ are designed appropriately.
Thus, we choose 
\begin{align*}
V_3(z_3) = 
z_3^\top 
\begin{bmatrix}
	7.80  &  4.87 \\
	4.87  &  4.82
\end{bmatrix}
z_3,
\text{ and }
\mathbb Z_3 =   \{\|z_3\|_\infty \leq 0.19\},
\end{align*}
which ensures closed-loop stability of \eqref{eq:sysDyn} \cite{Scokaert1999,Mayne2000}.

Now, we apply two variants of  Algorithm~\eqref{alg:appDP} to OCP \eqref{eq:OCP} for the above system dynamics and cost.
In the first variant, we use exact polytope projection from \autoref{sec:Pexact} for all time steps. 
For the second one, we use exact polytope projection for time steps $i\in\{2,3\}$ and the ellipsoidal inner approximation via \eqref{eq:ellFitDet} for time step $i=1$.\footnote{
	Note that one can combine  exact projections with inner approximations since exact projections are also inner approximations.}
We use the crude value function approximation $\tilde V_i \equiv0$.
Recall that by Theorem~\autoref{thm:feasADPappr}, Algorithm~\ref{alg:appDP} is guaranteed to compute a feasible point despite these approximations.

\autoref{fig:ocpdp} shows the exact polyhedral set projections $\{\mathcal P_t\}_{t \in \mathbb T}$ and the ellipsoidal inner approximation $\tilde {\mathcal P_1}$ computed in the backward sweep of  Algorithm~\ref{alg:appDP}.
Moreover, the optimal trajectory  $\{\Ae{z}_t^\star\}_{i \in \mathbb T}$ for solving OCP~\eqref{eq:OCP} is depicted.
Furthermore, the suboptimal but feasible solution trajectories  $\{\bar z_t^\star\}_{t \in \mathbb T}$ are shown in red, for the exact projections, and in green for the inner approximation.
One can see that, as predicted by Theorem~\autoref{thm:feasADPappr}, Algorithm~\ref{alg:appDP} indeed generates a feasible solution, which is suboptimal,  cf. the sequence of optimal coupling variables $\{z_t^\star\}_{t\in \mathbb T}$ vs. the sequence $\{\bar z_t^\star\}_{t\in \mathbb T}$.


\begin{remark}[Connection to existing techniques in MPC] \label{rem:exMPC}
	Notice that Algorithm~\ref{alg:appDP} can be considered as a generalization of methods, which are well-established in the control literature.
	Specifically, the backward sweep in Algorithm~\ref{alg:appDP} for exact computation of $\mathcal P_i$ with $V_i \equiv 0$ is equivalent to computing the \emph{$N$-step predecessor set} or the \emph{$N$-step backward reachable} set, cf. \cite{Angeli2008,Rakovic2006,Bjornberg2006,Blanchini1999,Bertsekas1972}.\hfill $\square$ 
\end{remark}

\subsection{Power Systems} \label{sec:PowerSys}
Next, we illustrate how Algorithm~\ref{alg:appDP} can be used for complexity reduction in optimization-based operation of power systems. 
Specifically, we apply our method to an AC Optimal Power Flow~(OPF) problem \cite{Frank2016} among multiple voltage levels, which is one of the most important optimization problem for power system operation.
In particular, we  show how to apply the above methods to compute  $\tilde {\mathcal  P}_i$,  as this is the most difficult step in Algorithm~\ref{alg:appDP}.

%

\subsubsection*{Grid Models}
Consider an electrical power system described by a graph  $G^e = (\mathcal N, \mathcal B)$, where $\mathcal N = \{1,\dots,|\mathcal N|\}$ is the set of buses and $\mathcal B \subseteq \mathcal N \times \mathcal N$ is the set of branches.
A  branch  models transmission lines, cables and transformers.
Consider a balanced grid with zero line-charging capacitances.
Then, the associated bus admittance matrix $Y = G+jB \in \mathbb{C}^{|\mathcal N|\times |\mathcal N|}$ is defined by
\[
[Y]_{k,l}=
\begin{cases}
	\sum_{k \in \mathcal N }y_{k,l} & \mbox{if} \quad k=l, \\
	-y_{k,l}, & \mbox{if} \quad k \neq l.
\end{cases}
\]
Here, $y_{k,l} = g_{k,l} + j b_{k,l}\in \mathbb{C}$ is the admittance of  branch $(k,l) \in \mathcal{E}$, where $g_{k,l}$ and $b_{k,l}$ denote its susceptance and conductance, respectively.
Note that $y_{k,l} =0$ if $(k,l) \notin \mathcal B$.
The  flow of active power $p_{k,l}$ and reactive power $q_{k,l}$ along branch $(k,l) \in \mathcal{E}$ is given by
\begin{subequations} \label{eq:lineFlows}
	\begin{align}
		p_{k,l} &= v_kv_l(G_{k,l}\cos(\theta_{k,l})+B_{k,l}\sin(\theta_{k,l})), \\
		q_{k,l} &= v_k v_l(G_{k,l}\sin(\theta_{k,l})-B_{k,l}\cos(\theta_{k,l})). \label{eq:PFEQq}
	\end{align}
\end{subequations}
Here, $v_k$ is the voltage magnitude at vertex $k\in \mathcal{N}$ and $\theta_{k,l}\doteq \theta_k - \theta_l$ is the voltage angle difference along the branch $(k,l) \in \mathcal{E}$.
With the above equations, we  model the grid behavior by  the AC power flow equations for all buses $k \in \mathcal{N}$
\begin{align} \label{eq:PFeq}
	&p_k = \sum_{l \in \mathcal{N}} p_{k,l} , \qquad  q_k = \sum_{l \in \mathcal{N}} q_{k,l},
\end{align}
where $p_k, q_k \in \mathbb{R}$ are the net active and reactive power injection at vertex $k \in \mathcal N$,  i.e. the residual of all generation and all demand.
The net powers are given by 
\begin{align} \label{eq:netPwr}
	p_k &= p_k^g   - p_k^d \qquad 	q_k = q_k^g - q_k^d,
\end{align}
where $p_k^g$ and $q_k^g$ are the active/reactive power generation, and $p_k^d$ and $q_k^d$ are the active/reactive power demand at bus $k \in \mathcal N$.

Under the assumption of zero line resistances, small voltage angles and constant voltage magnitudes,  \eqref{eq:lineFlows} can be simplified to the so-called DC model
\begin{align} \label{eq:DCPF}
	p_{k,l} = B_{k,l} \theta_{k,l}
\end{align}
for all $(k,l) \in \mathcal B$, cf.  \cite{Frank2016}.


%
%

\subsubsection*{Optimal Power Flow}
The goal of OPF is to minimize the cost of power generation while satisfying all grid constraints.
Define  a vector of controls $u_{\mathcal N} \doteq \operatorname{vcat}(\{(p_k^g,q_k^g)\}_{k \in \mathcal N}) \in \mathbb{R}^{2|\mathcal N|}$ and a vector of algebraic states 
$x_{\mathcal N}^\top  \doteq\operatorname{vcat}\big ( \operatorname{vcat}(\{[v_k,\theta_k]^\top\}_{k \in \mathcal N})]^\top, \operatorname{vcat}(\{(p_{k,l},q_{k,l})^\top\}_{(k,l) \in \mathcal B})   \big )$.
Then, a basic AC OPF problem reads
	\begin{subequations} \label{eq:ACOPF}
		\begin{align}
			&\min_{u_{\mathcal N},x_{\mathcal N}} \sum_{k \in \mathcal{G}} c_k( p_k^g)^2 + d_k p_k^g  \\
&			\text{subject to } \;\; \eqref{eq:lineFlows}, \eqref{eq:PFeq},  \eqref{eq:netPwr}, \hspace{-2cm}
			&\hspace{-2cm} \label{eq:prevCnstr}\\
			&\underline {v} \leq v_k \leq \bar{v}, &&& k &\in \mathcal{N}, \label{eq:vCstr} \\
			&p^g_k = q^g_k = 0, &&&k& \in \mathcal{N} \setminus \mathcal G, \hspace{-.5em} &  \label{eq:genCstrZero}\\
			&\underline p_k^g \leq p_k^g \leq \bar p_k^g&&& k& \in  \mathcal G, \label{eq:genCstrBds} \\
			&p_{k,l}^2 + q_{k,l}^2 \leq \bar s_{k,l}^2, \hspace{-1em} &&& (k,l) &\in \mathcal{E}, \label{eq:lineLim} \\
			&p_{k}^2 + q_{k}^2 \leq \bar s_{k}^2, \quad
			 -p_k\leq \alpha q_k\leq p_k, \hspace{-2em}  &&& k &\in \mathcal{G}, \label{eq:ice_cream_ratio1}\\
			&\theta_1 =0,  \;\;v_1=1, \hspace{-1cm}  \label{eq:refConstr}
		\end{align}
	\end{subequations}
where $\mathcal G \subseteq \mathcal N$ is a set of generator buses and \eqref{eq:genCstrBds} are lower and upper bounds on its power generation.
	Here, $\{(c_k,d_k)\}_{k \in \mathcal{G}}$ are generator-specific cost coefficients and \eqref{eq:refConstr} are  reference constraints, cf. \cite{Frank2016}. 
	Note that the power demands $\{p_k^d\}$ are given and fixed.
	The constraint \eqref{eq:vCstr} expresses bounds on the voltage magnitudes and \eqref{eq:lineLim} considers a simplified current constraint over branch $(k,l) \in \mathcal{E}$.
		The constraint \eqref{eq:ice_cream_ratio1} combines an upper bound to the maximum apparent power  with a  power factor constraint  $\alpha > 0$~\cite{contreras2018}.
%
%

	\TF{With the simplified  grid model from \eqref{eq:DCPF}, the DC counterpart to \eqref{eq:ACOPF} reads}
	\begin{subequations} \label{eq:DCOPF}
		\begin{align}
			\min_{u_{\mathcal N},x_{\mathcal N}}& \sum_{k \in \mathcal{N}} c_k( p_k^g)^2 + d_k p_k^g  \\
			\text{s.t.} \;\; & \eqref{eq:PFeq},  \eqref{eq:netPwr},\eqref{eq:DCPF},\;\; \theta_1 =0,\hspace{-2cm}
			&\hspace{-2cm} \label{eq:prevCnstrDC}\\
				&p^g_k = 0, \;\;&k &\in \mathcal{N} \setminus \mathcal G,   \\
			&\underline p_k^g \leq p_k^g \leq \bar q_k^g, &k& \in  \mathcal G, \label{eq:genCstrBdsDC} \\
			-&\bar p_{k,l}\leq p_{k,l} \leq \bar p_{k,l}, & (k,l) &\in \mathcal{E}, \label{eq:lineLim2} 
		\end{align}
	\end{subequations}
	where \eqref{eq:lineLim2} is a simplified version of the line limit \eqref{eq:lineLim}.
	
	\subsubsection{Exploiting Network Structure} \label{sec:NetStrc}
	\begin{figure}
	\centering
	\includegraphics[width=0.48\linewidth,trim={3cm 7cm 3cm 4cm}]{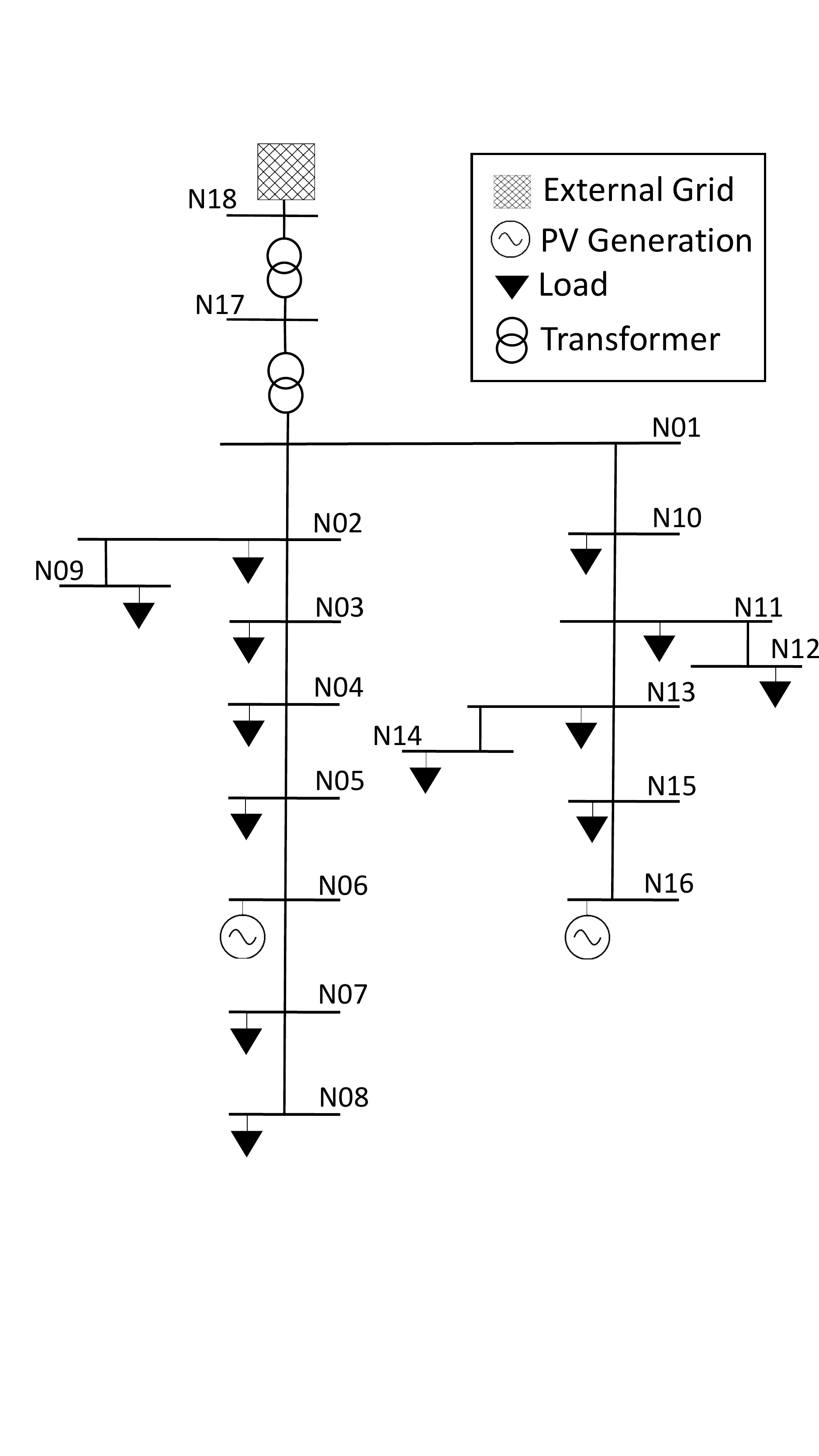}
	\caption{Modified distribution grid model from \cite{Grady1992}. }
	\label{fig:grid}
\end{figure}	
	For applying Algorithm~\ref{alg:appDP}, we first have to \Ae{reveal the tree-strucure of OPF problems} \eqref{eq:ACOPF} and \eqref{eq:DCOPF}, \Ae{cf. \autoref{ass:tree}.}
 \Ae{A typical} distribution grid is shown in \autoref{fig:grid}.
	Observe that this grid has no connection to other distribution grids, which leads to a tree structure in the interaction graphs of  \eqref{eq:ACOPF}, \eqref{eq:DCOPF} if the vertex set $\mathcal N$ is partitioned appropriately. 	\Ae{This structure is  very common in power systems \cite{Weedy2012,Short2004} and it can be seen as a cornerstone assumption for most ``aggregation'' approaches \cite{MayorgaGonzalez2018,capitanescu2018, kalantar2020,silva2018,contreras2018,nazir2022,givisiez2020,Tan2023a,Tan2019,Zhao2016,Muller2019a, Kellerer2015,Capitanescu2023}, which we would like to generalize with our work.} 
	Here, for the sake of simplicity, we consider a tree-depth of $2$, i.e. , we only have one upper-level (transmission) grid and several lower-level (distribution) grids connected to it.

	Partition the set of buses $\mathcal{N}$ into a subset containing the buses of the upper-level grid  $\mathcal{N}_1$ and into $|\mathcal{S}|-1$ subsets for the lower grid levels $\mathcal N_i, i \in \mathcal{S} \setminus \{1\}$. Where $\mathcal S =\{1,\dots,|\mathcal S|\}$ is such that we have $\left (\cup_{i \in \mathcal{S} } \mathcal N_i \right ) = \mathcal{N}$ and all $\mathcal N_i$ are mutually disjoint.
	\Ae{Moreover, denote $\mathcal N_i^c \subseteq \mathcal N_1$ as the set of buses in the upper level to which the $i$th lower-level grid  is connected.}
	Similarly, 	 partition the set of branches $\mathcal B$ into $\{\mathcal B_i\}_{i \in \mathcal S}$ such that all branches connecting vertexs in $\mathcal N_i$ belong to $\mathcal B_i$, where we assign the branches connecting lower and upper-level grid $\mathcal B_i^c \subseteq \mathcal B_i$  to the \Ae{lower} level.
	This way, we \Ae{define
	\begin{align*}
&	\quad	f_i(x_{{\mathcal I}_i})\doteq \sum_{k \in \mathcal{G} \cap \mathcal N_i} c_k( p_k^g)^2 + d_k p_k^g + \iota_{\mathcal X_i}(x_{{\mathcal I}_i}),\; \text{ where }\\
&	 \Ae{ \mathcal X}_i \doteq \big \{ \Ae{x_{\mathcal I_i} }\;|\;\eqref{eq:prevCnstr}-\eqref{eq:lineLim} \text{ hold } \forall k \in \mathcal N_i,  \forall (k,l) \in \mathcal B_i \Ae{\cup \mathcal B_i^c}\big \}, 
	\end{align*}}
	for all $i \in \mathcal S \setminus \{1\}$.
	 $ \Ae{\mathcal X}_1$ is defined as above with the additional reference constraints \eqref{eq:refConstr} \Ae{and neglecting $\mathcal B_i^c$.}
	According to the definitions of local and coupling vectors from \autoref{sec:probForm}, this leads to local decision variables  $\Ae{x_{\mathcal I_i}}^\top =\big   [\operatorname{vcat}(\{[p^g_k, q^g_k, p_k, q_k, v_k, \theta_k]\}_{k\in \mathcal N_i \Ae{\cup \mathcal N_i^c}}),$ $ \operatorname{vcat}(\{[p_{k,l}, q_{k,l}]\}_{(k,l) \in \mathcal B_i \cup \mathcal B_i^c}^\top )\big ]$ for all $i \in \mathcal S\setminus \Ae{\{1\}}$ and \Ae{$x_{\mathcal I_1}$ is defined in the same way but excluding the active and reactive power flows $(p_{k,l},q_{k,l})$ over coupling lines $(k,l) \in \mathcal B_i^c$.
	}
	
	\Ae{Next, we identify coupling variables.}
	\Ae{Note that the only objective/constrain functions coupling the upper-level grid and the lower level grids are the line flow equations \eqref{eq:lineFlows} for coupling lines $\mathcal B_i^c$.
		Since these constraints are assigned to the lower-level grids, and the decision variables $(p_{k,l},q_{k,l})$ appear also in the upper-level grid through the power balance constraint \eqref{eq:PFeq}, we identify them as coupling variables.
		Moreover, the pair $(v_k,\theta_k)$ at the coupling vertexs in the upper-level $k\in \mathcal N_i^c$ also appear in this constraint.
		Hence, we have four coupling variables per coupling vertex $k\in \mathcal N_i^c$
		\begin{align*}
			\Ae{x_{\mathcal W_i}}^\top \doteq [\Ae{\operatorname{vcat}(\{v_{k}, \theta_{k}\}_{k \in \mathcal N_i^c  })^\top}, \operatorname{vcat}(\{p_{k,l}, q_{k,l}\}_{(k,l) \in \mathcal B_i^c  })^\top ].
		\end{align*}
	}

\begin{figure*}[h]
		\begin{subfigure}[t]{0.33\textwidth}
		\centering
		\includegraphics[width=.75\textwidth]{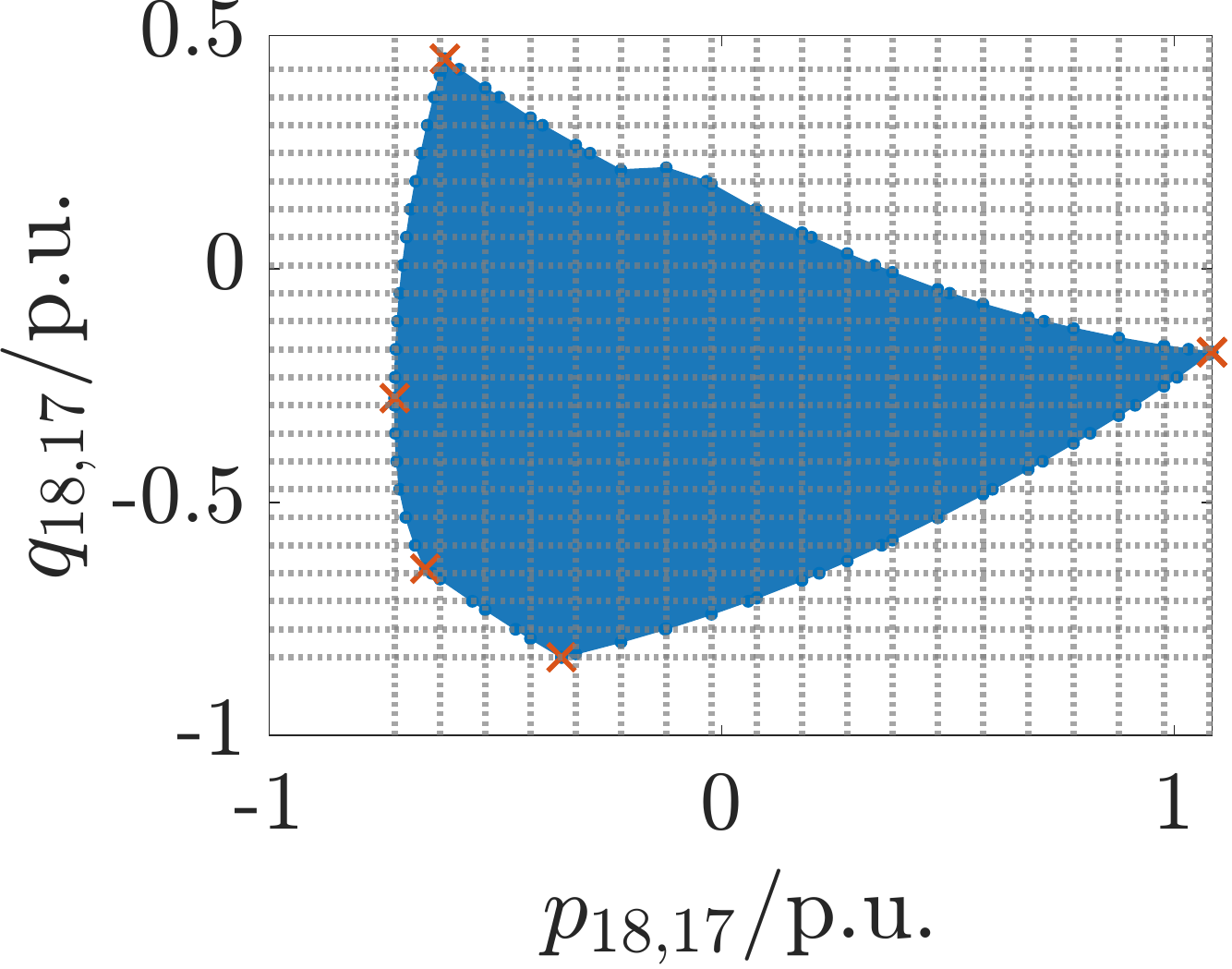}
		\caption{$\mathcal P_2$ \Ae{for $v_{18}=1\,$p.u..}}
		\label{fig:aggregation}
	\end{subfigure}
	\begin{subfigure}[t]{0.33\textwidth}
		\centering
		\includegraphics[width=.95\textwidth]{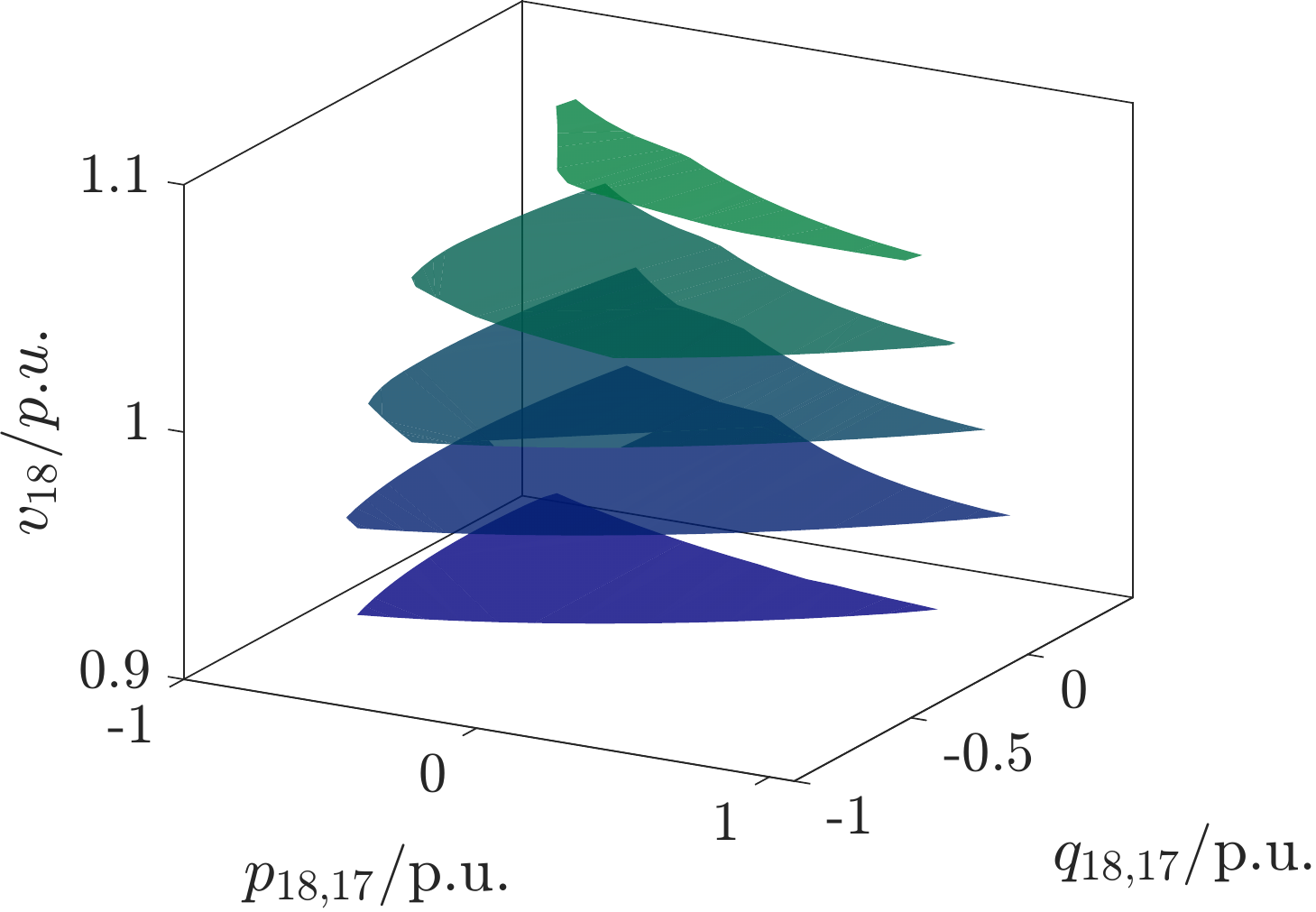}
		\caption{Influence of $v_{18}$ on ${\mathcal{P}_2}$.}
		\label{fig:v_influence}
	\end{subfigure}
	\begin{subfigure}[t]{0.33\textwidth}
		\centering
		\includegraphics[width=.8\textwidth]{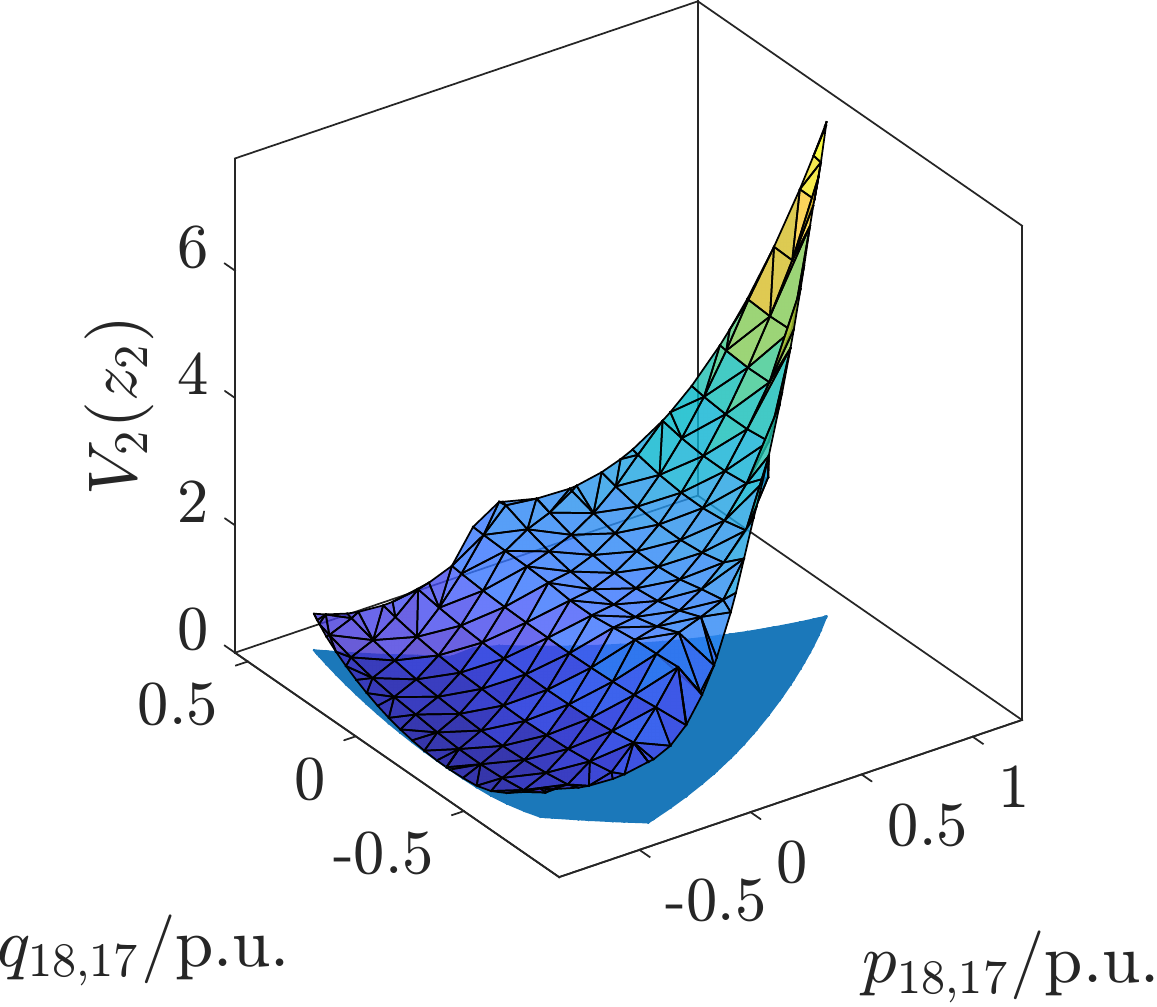}
		
		\caption{$V_2$ and $\mathcal P_2$  \Ae{for $v_{18}=1\,$p.u..}}
		\label{fig:ValFun}
	\end{subfigure}
	\caption{\Ae{Projection  $\mathcal P_2$ and  value function $V_2$ for the AC OPF problem  \eqref{eq:ACOPF} with data from \autoref{fig:grid}.}}
\end{figure*}

	\subsubsection*{\Ae{Fostering Numerical Tractability}}
	\Ae{Since we consider only one coupling vertex $|\mathcal N_i^c|=1$  and   since \eqref{eq:lineFlows} depends only on the voltage angle difference $(\theta_k-\theta_l)$, we have $p_{k,l}(v_k,v_l, \theta_k,\theta_l) = p_{k,l}(v_k,v_l, \theta_k + \bar \theta,\theta_l+\bar \theta)$ for any $\bar \theta \in \mathbb R$.
		Thus fixing $\theta_k = 0$ in $ \mathcal X_i$  does not alter the solution space
		}
	\Ae{	Hence, we can neglect the coupling from $\theta_k, k \in \mathcal N_i^c$ and obtain a reduced set of coupling variables\footnote{\Ae{Note that this is not possible in case of multiple grid interconnection  points, since in this case there is more than one $\theta_k^\star$ in $\mathcal N_i^c$ and these $\theta_k^\star$ can not be chosen independently. The efficient handling of  multiple interconnections is subject to ongoing research \cite{Majumdar2023}.}}
		\begin{align*}
			\Ae{\hat x_{\mathcal W_i}}^\top \doteq [\Ae{\operatorname{vcat}(\{v_{k}\}_{k \in \mathcal N_i^c  })^\top}, \operatorname{vcat}(\{p_{k,l}, q_{k,l}\}_{(k,l) \in \mathcal B_i^c  })^\top ] \in \mathbb R^{3}.
	\end{align*}}
	\subsubsection{Numerical Results}	\label{sec:numResOPF}

	Consider  the distribution grid from \autoref{fig:grid}, which is based on \cite{Grady1992}.
	In this model, we set $\mathcal G = \{6,16\}$, which are equipped with (curtailable) solar generators with a high peak-power.
	The distribution grid is connected to the transmission grid, where we consider bus $18$ as a bus of the transmission grid.
	Label the distribution grid with $i=2$.
	Hence, we have $\Ae{x_{\mathcal W_2}} = [\Ae{v_{18}},p_{18,17}, q_{18,17}]$.
	\TF{Next, we compute $\mathcal P_2$ and $V_2$ for the AC problem \eqref{eq:ACOPF} and its DC simplification \eqref{eq:DCOPF}.}
All  values are normalized via the per-unit (p.u.) system in the following, cf. \cite[Appendix C]{Frank2016}.

	\subsubsection*{AC OPF}

	The AC OPF problem \eqref{eq:ACOPF} is nonconvex.
	Hence, we rely on  \autoref{alg:discMeth2} to compute a set of sample points $\{\mathsf z_l\}_{l\in \mathcal T}$ at the boundary of $\mathcal P$, \Ae{where we fix certain values for $v_k,k\in \mathcal N_i^c$ and then solve~\eqref{eq:boundaryProb} to obtain an approximation of the three-dimensional solution space.}
\Ae{The  result is shown in  Figure \ref{fig:v_influence}.
One can  see that the value of $v_k$ has a strong influence on the admissible $p_{18,17}/q_{18,17}$-area.}
	\Ae{Fixing the value of $v_{18}$  to a certain value  is a common simplification in aggregation approaches for power systems.\footnote{\TF{With the exception of \cite{Sarstedt2020}, works on flexibility aggregation do not consider $v_k$ as a coupling variable, but  instead use a constant value usually such as 1 p.u., assuming the voltage drop in the transmission grid is negligible. The works \cite{purchala2005,stott2009} underline the practical value of this commonly used assumption while acknowledging its occasional limitations.} }
	Doing so yields a two-dimensional coupling space.
	The resulting $\mathcal P$ for  $v_{12}=1\,\mathrm{p.u,}$ is shown in \autoref{fig:aggregation}, where the computed samples are displayed as red crosses.}
	Since using these points only would be too coarse, we use  gridding  from \autoref{sec:PinnerAppox} to obtain a finer discretization of $\mathcal P$.
	Here, the interval $-0.72\leq p_{18,17} \leq 1.08$ is discretized via 19 equi-distant discretization points, and \eqref{eq:boundaryProb} is solved for these values fixed for $c_l\in \{-1,1\}$ yielding minimal/maximal values of $q_{18,17}$ at these discretization points for $p_{18,17}$.
	The same approach is applied to $q_{18,17}$. 
	The resulting  points are shown as blue dots in \autoref{fig:aggregation}.
	Observe that $\mathcal P_1$ is nonconvex. 
	Hence, for applying \autoref{alg:appDP}, one has to construct inner approximations $\tilde {\mathcal P_2} \subset \mathcal P_2$ in order to guarantee feasibility.

	\autoref{fig:ValFun} shows the corresponding \textit{true} value function $V_2(z_2)$ for all $z_2\in \mathcal P_2$ obtained by discretizing the interior of $\mathcal P_2$ and evaluating $V_2$.
	Observe that, by definition, $V_2(z_2)=\infty$ for all $z_2 \notin \mathcal P_2$.
	One can see, as mentioned in \autoref{rem:propV}, $V_2$ is nonlinear and nondifferentiable; hence one would have to rely on approximations $\tilde V_2 \approx V_2$ in order to communicate and efficiently integrate this function in the optimization for the transmission grid.

%

	\subsubsection*{DC OPF}
Observe that the feasible set of the DC OPF problem \eqref{eq:DCOPF} is a convex polyhedron in form of \eqref{eq:polCstr}.
Hence, we can rely on an exact polytope projection.
Since the reactive power $q_{18,17}$  is neglected in the DC model, $\mathcal P_2$ becomes an interval yielding $\mathcal P_2 = \{-0.76\,\text{p.u.} \leq p_{18,17} \leq 1.04\,\text{p.u.} \}$.
The corresponding value function is shown in \autoref{fig:dc_value}.
Observe that also here, $V_2$ is nondifferentiable.
However, since \eqref{eq:DCOPF} is a quadratic program, $V_2$ is piecewise-quadratic~\cite{Bemporad2000}.

	\begin{figure}[t]
	\centering
	\includegraphics[width=0.45\linewidth]{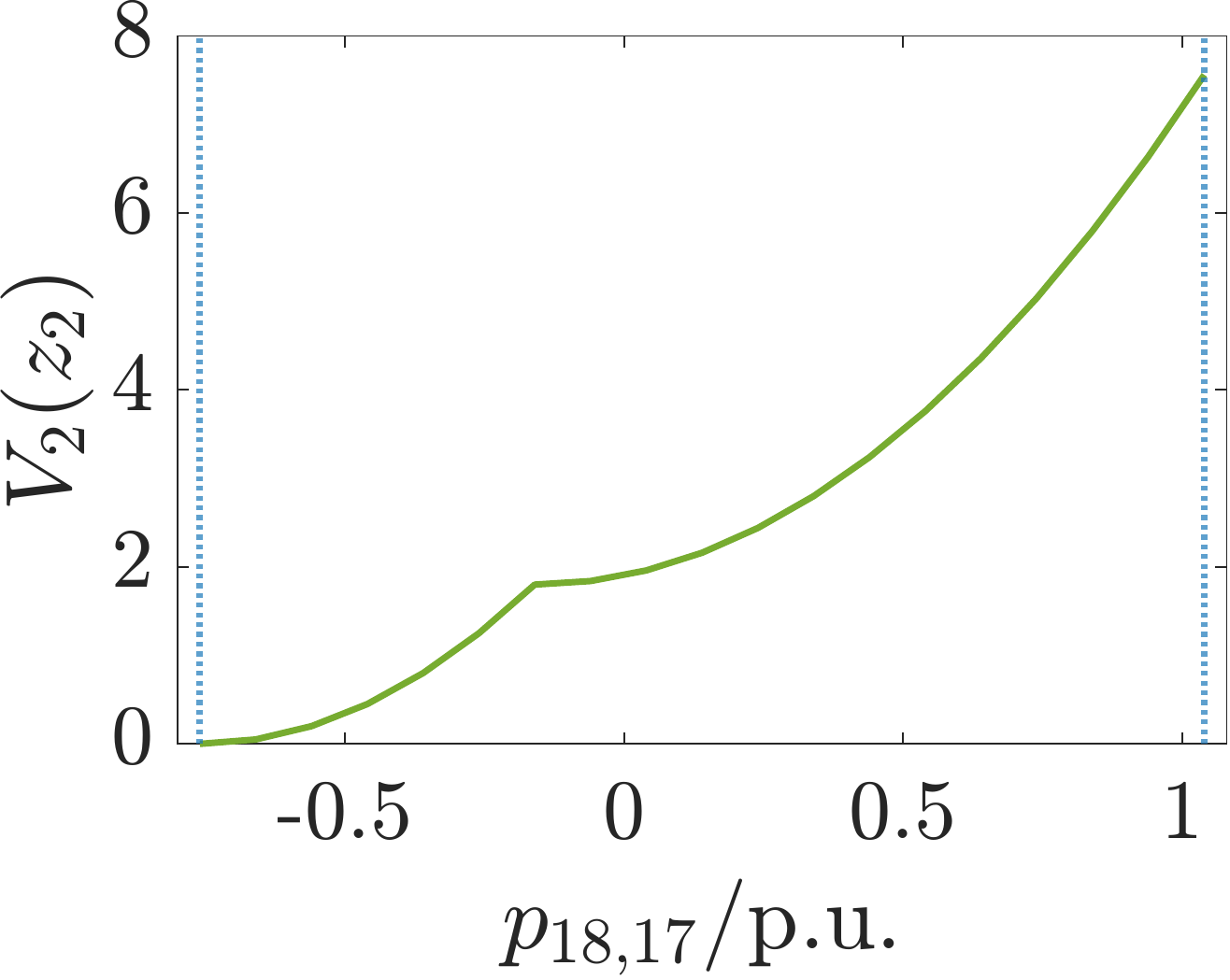}
	\caption{$V_2$ and $\mathcal P_2$ for the DC OPF problem \eqref{fig:grid}.}
	\label{fig:dc_value}
\end{figure}

\Aee{Note that for polytopic constraint sets from the DC model or LinDistFlow \cite{Baran1989}, the projection scales up to moderate dimensions of $\mathcal P_i$. 
}

\begin{remark}[Connection to  methods from power  systems] \label{rem:exPsys}
Algorithm~\ref{alg:appDP} is a generalization of methods, which have been recently developed for solving special instances of AC OPF problems,  for charging of electric vehicles and for the  integration of storage devices \cite{Krause2009,capitanescu2018, kalantar2020,silva2018,MayorgaGonzalez2018,Evans2022,Angeli2021,Ozturk2022,Appino2021}.
Here, the backward sweep computing $\mathcal P_i$ is called \emph{aggregation} and the forward sweep is called \emph{disaggregation} or \emph{dispersion}.
This naming comes from the fact that the degrees of freedom in the lower grid levels are seen as \emph{flexibility}, i.e., they represent flexible load or generation, which can be used for avoiding congestions in the upper grid levels by solving the corresponding upper-level problems.
\Ae{Recent works have  recognized the connection between aggregation and  projection for the special case of polyhedral constraint sets \cite{Tan2023a,Zhao2016}.}
For an overview on methods in this domain we refer to \cite{Sarstedt2021}. \hfill $\square$
\end{remark}

%
%

	
	\section{Conclusions} \label{sec:conclusions}
	This paper has presented an approximate dynamic programming scheme for complexity reduction in hierarchical tree-structured optimization problems of networked systems.
	Our approach guarantees the computation of feasible but potentially suboptimal points in \emph{one shot}, which is of particular importance in many engineering contexts.
	\TF{Moreover, our proposed scheme appears to be  a generalization of methods, which have been used in different domains. Hence, it  provides an avenue for  method transfer between domains. }
	
	The proposed scheme relies on the ability to compute set projections or inner approximations thereof.
	Thus, future work will focus on the further development of these methods.

\appendices

\section{Proofs} \label{sec:prelim}
\Ae{
	
\textit{Proof of  \autoref{lem:treeRef}:}
		Start at a leaf $l \in \mathcal L$ of the interaction tree $G$.
		The parent of subsystem $l$ is unique by Assumption~\ref{ass:tree} \cite[Cor 1.5.2]{Diestel2017}. Thus, $x_{{\mathcal W}_l}$ is well-defined. Hence, by separability of  \eqref{eq:sepNLP1} and since $\min_{x,y} f(x) + g(y) = \min_{x} f(x) + \min_y g(y)$, we can  define $V_l(x_{{\mathcal W}_l}) \doteq \min_{x_{{\mathcal L}_l}} f_l(x_{{\mathcal L}_l})$ at all leaves.
		This matches definition \eqref{eq:ValFunRef0} since $\mathcal C_l=\emptyset$ and summing over an empty set is zero by convention.
		At  parents $j \in \operatorname{par}(l)$, we  apply definition \eqref{eq:ValFunRef0} replacing the index $j$ with $l$.
		By induction, we proceed recursively until the root-subsystem $1$ is reached for which $\operatorname{par}(1)=\emptyset$ and thus $\mathcal{W}_1 = \emptyset$.
		Hence, at the root, $V_1$ has no arguments and thus \eqref{eq:ValFunRef0} reduces to an optimization problem without arguments of the value function.}
	\hfill$\blacksquare$

\textit{Proof of  \autoref{lem:domVi}:}
	\Ae{	Consider a point $\bar z \in \operatorname{dom}(V)$, i.e. $V(\bar z) < \infty$. 
		Thus, by \eqref{eq:valFunProb3}, there must exist a corresponding $\bar y$ such that $F(\bar y,\bar z)< \infty$, i.e. a  $\bar y$ such that $(\bar y,\bar z)\in \operatorname{dom} (F)$.
		Hence, $\bar z \in \operatorname{proj}_{\mathcal Z}{(\operatorname{dom}(F))}$ by definition of $\operatorname{proj}_{\mathcal Z}$.
		
		Consider a point $\bar z \in  \operatorname{proj}_{\mathcal Z}{(\operatorname{dom}(F))}$.
		By  definition of the projection, there  exists a corresponding $\bar y$ such that $(\bar y, \bar z ) \in \operatorname{dom}(F)$.
		Hence, $V(\bar z) < \infty$ in \eqref{eq:valFunProb3}, thus $\bar z \in \operatorname{dom}(V)$.}
\hfill$\blacksquare$
\printbibliography

\begin{IEEEbiography}
	[{\includegraphics[width=1in,height=1.25in,clip,keepaspectratio]{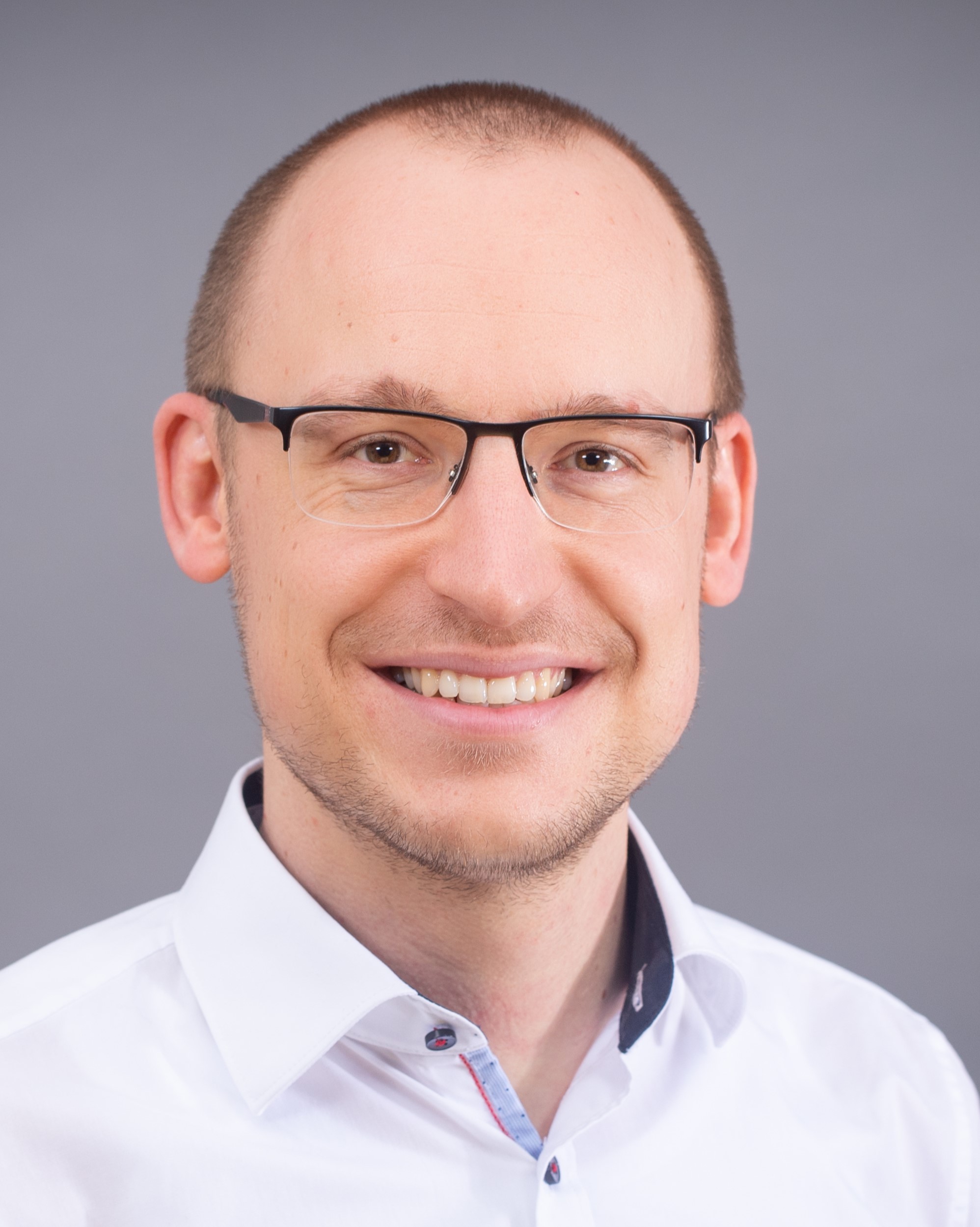}}]{Alexander  Engelmann}  received the M.Sc. degree in electrical engineering and information technology and the Ph.D. degree in informatics from the Karlsruhe Institute of Technology, Karlsruhe, Germany, in 2016 and 2020, respectively.
	Since 2020, he is a postdoctoral researcher at the Institute for Energy Systems, Energy Efficiency and Energy Economics at TU Dortmund University, Dortmund, Germany. His research focuses on distributed optimization and optimal control for power and energy systems.
\end{IEEEbiography}

\begin{IEEEbiography}
	[{\includegraphics[width=1in,height=1.25in,clip,keepaspectratio]{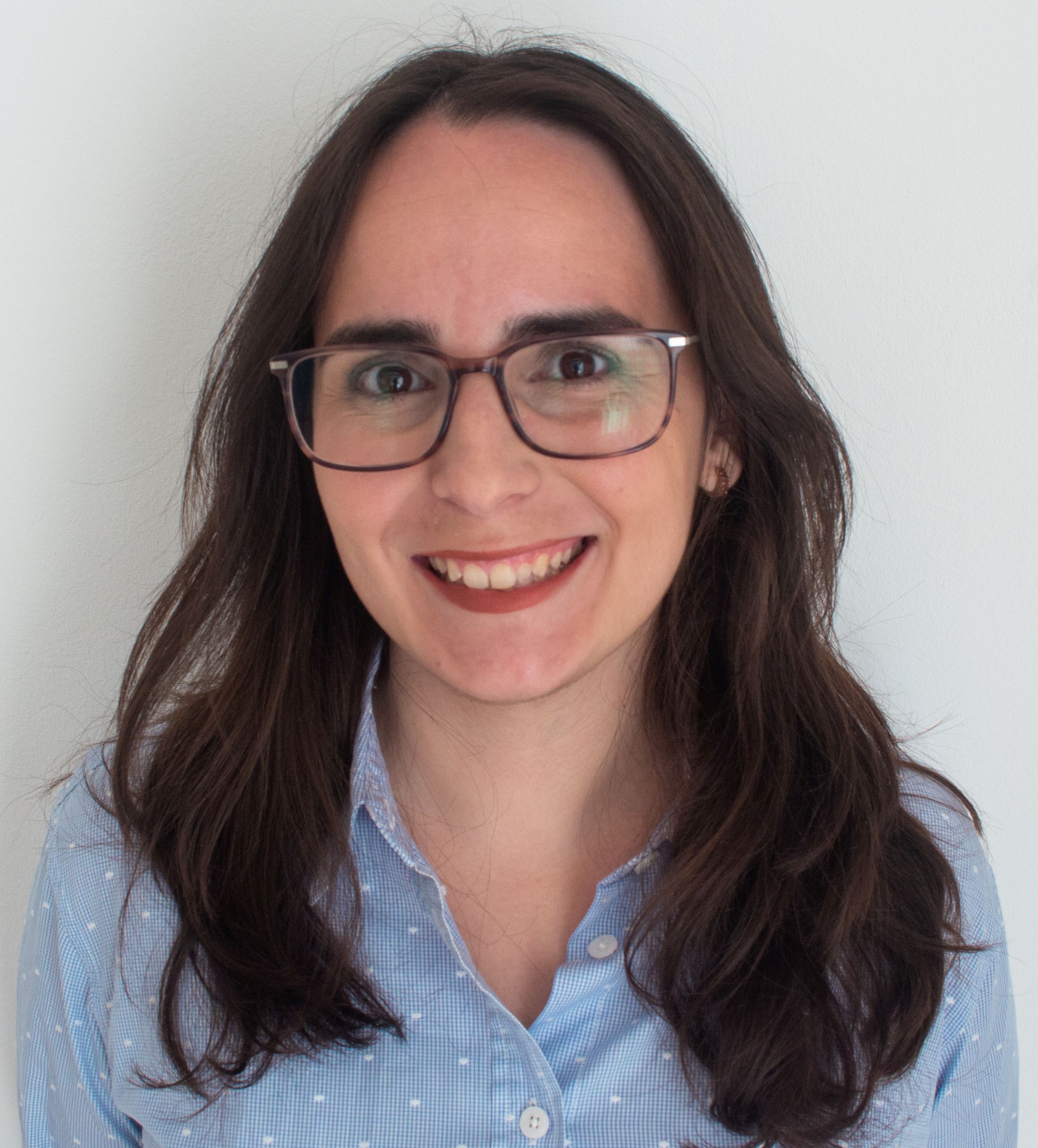}}]{Maísa Beraldo Bandeira} studied Control and Automation engineering at Universidade Federal de Santa Catarina in Brazil. In 2022, she received the M.Sc. degree and started pursuing the Ph.D. degree both in Electrical Engineering and Information Technology at TU Dortmund University. Her research focuses on power aggegation methods for multi voltage level coordination.
\end{IEEEbiography}

\begin{IEEEbiography}
	[{\includegraphics[width=1in,height=1.25in,clip,keepaspectratio]{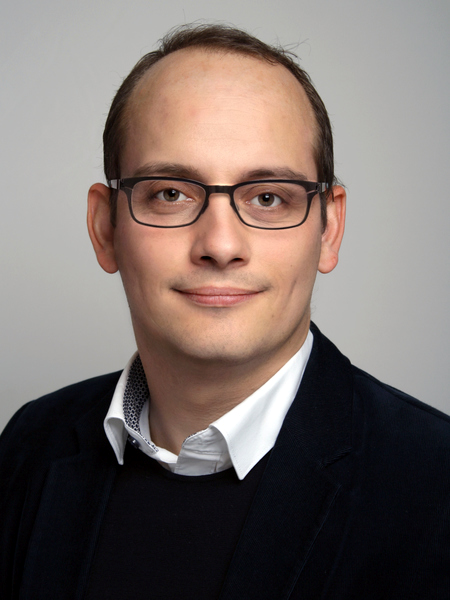}}]{Timm Faulwasser }  has studied Engineering Cybernetics at the University of Stuttgart, with majors in systems and control and philosophy. From 2008 until 2012 he was a member of the International Max Planck Research School for Analysis, Design and Optimization in Chemical and Biochemical Process Engineering Magdeburg. In 2012 he obtained his PhD from the Department of Electrical Engineering and Information Engineering, Otto-von-Guericke-University Magdeburg, Germany. From 2013 to 2016 he was with the Laboratoire d’Automatique, École Polytechnique Fédérale de Lausanne (EPFL), Switzerland, while 2015-2019 he was leading the Optimization and Control Group at the Institute for Automation and Applied Informatics at Karlsruhe Institute of Technology (KIT), where he successfully completed his habilitation in the Department of Informatics in 2020. In November 2019 he joined the Department of Electrical Engineering and Information Technology at TU Dortmund University, Germany. Currently, he serves as associate editor for the IEEE Transactions on Automatic Control, the IEEE Control System Letters, as well as Mathematics of Control Systems and Signals. His main research interests are optimization-based and predictive control of nonlinear systems and networks with applications in energy, process systems engineering, mechatronics, and beyond. Dr. Faulwasser received the 2021-2023 Automatica Paper Prize. 
\end{IEEEbiography}

\end{document}